\documentclass{amsart}
\usepackage{graphicx}
\usepackage{comment}
\def\R{\mathbb R}
\def\Z{\mathbb Z}

\def\Q{\mathbb Q}
\def\C{\mathbb C}
\def\HH{\mathcal H}

\def\AA{\mathcal A}

\def\la {{\lambda}}

\newcommand {\nc}   {\newcommand}
\nc {\be}   {\begin{equation}} \nc {\ee}   {\end{equation}} \nc
{\beq}  {\begin{eqnarray}} \nc {\eeq}  {\end{eqnarray}} \nc {\beqs}
{\begin{eqnarray*}} \nc {\eeqs} {\end{eqnarray*}}
\def\edc{\end{document}}

\newtheorem{theoreme}{Theorem}[section]
\newtheorem{pro}[theoreme]{Proposition}
\newtheorem{lemma}[theoreme]{Lemma}

\theoremstyle{definition}

\newtheorem{rem}[theoreme]{Remark}

\numberwithin{equation}{section}
\usepackage{hyperref}
\hypersetup{
   bookmarks,%
                colorlinks,%
                urlcolor=blue,%
                citecolor=blue,%
                linkcolor=blue,%
                hyperfigures,%
                pagebackref,%
                pdfcreator=LaTeX,%
                breaklinks=true,%
                pdfpagelayout=SinglePage,%
                bookmarksopen=true,%
                bookmarksopenlevel=2         
                }
\begin{document}
\title[The  influence of the coefficients on the stability of coupled  wave equations]{The influence of the coefficients of a system of wave equations coupled by velocities on its stabilization}
\author{Mohammad AKIL}
\address{Universit\'e Libanaise\\
KALMA Laboratory, Equipe EDP-AN\\
Hadath, Beyrouth, Liban}
\email{mohamadakil1@hotmail.com}
\author{Mouhammad GHADER}
\address{Universit\'e Libanaise\\
KALMA Laboratory, Equipe EDP-AN\\
Hadath, Beyrouth, Liban}
\email{mhammadghader@hotmail.com}
\author{Ali Wehbe}
\address{Universit\'e Libanaise\\
Facult\'e des Sciences 1\\
EDST, Equipe EDP-AN\\
Hadath, Beyrouth, Liban}
\email{ali.wehbe@ul.edu.lb}

\date{\today}
\begin{abstract}
In this work, we consider a system of two wave equations coupled by velocities in one-dimensional space, with one boundary fractional damping. First, we show that the system is strongly asymptotically stable if and only if the coupling parameter $b$ of the two equations is outside a discrete set of exceptional real values. Next, we show that our system is not uniformly stable. Hence, we look for a polynomial decay rate for smooth initial data. Using frequency domain approach combining with multiplier method, we prove that the energy decay rate is greatly influenced by the nature of the coupling parameter $b$, the arithmetic property of the ratio of the wave propagation speeds $a$, the order of the fractional damping $\alpha$. Indeed, under the equal speed propagation condition {\it i.e} $a=1$, we establish an optimal polynomial energy decay rate of type $t^{-\frac{2}{{1-\alpha}}}$ if the coupling parameter $b\notin \pi \Z$ and of type $t^{-\frac{2}{{5-\alpha}}}$ if the coupling parameter $b\in \pi \Z$. Furthermore, when the wave propagate with different  speeds {\it i.e} $a\not=1$, we prove that, for any rational number $\sqrt{a}$ and almost all irrational number $\sqrt{a}$, the energy of our system decays  polynomially to zero like as $t^{-\frac{2}{{5-\alpha}}}$. This result still beholds if $a\in \Q$, $\sqrt{a}\notin \Q$ and $b$ small enough.
\end{abstract}
\maketitle
\section{Introduction}\label{Section-1}
\setcounter{equation}{0}
\noindent \subsection{The problem} In this work, we investigate  the energy decay rate of a coupled wave equations with only one fractional dissipation law. The system is described by
\begin{eqnarray}
u_{tt}-u_{xx}+by_t&=&0,\ \ (x,t)\in (0,1)\times \mathbb{R}_+^\star,\label{eq-1.1}\\ \noalign{\medskip}
y_{tt}-ay_{xx}-bu_t&=&0,\ \ (x,t)\in (0,1)\times \mathbb{R}_+^\star,\label{eq-1.2}
\end{eqnarray}
with the following boundary conditions
\begin{eqnarray}
u(0,t)=y(0,t)=y(1,t)=0,\ \ t\in \mathbb{R}_+^\star,\label{eq-1.3}\\ \noalign{\medskip}
u_x(1,t)+\gamma\partial^{\alpha,\eta}_tu(1,t)=0,\ \ t\in \mathbb{R}_+^\star,\label{eq-1.4}
\end{eqnarray}
and the following initial conditions
\begin{eqnarray}
u(x,0)=u_0(x),& u_t(x,0)=u_1(x),&x\in    (0,1),\label{eq-1.5}\\
y(x,0)=y_0(x),& y_t(x,0)=y_1(x),&x\in (0,1),\label{eq-1.6}
\end{eqnarray}
where $\eta \geq 0$, $\alpha\in ]0,1[$, $a>0$ and $b\in \R^{\ast}$ are constants. Fractional calculus includes various extensions of the usual definition of derivative from integer to real order, including the Riemann-Liouville derivative, the Caputo derivative, the Riesz derivative, the Weyl derivative, etc. In this paper, we consider the Caputo's fractional derivative $\partial_t^{\alpha,\eta}$ of order $\alpha\in ]0,1[$ with respect to time variable $t$  defined by
\begin{equation}\label{eq-1.7}
\left[D^{\alpha,\eta}\omega \right](t) =\partial_t^{\alpha,\eta}\omega(t)=\frac{1}{\Gamma(1-\alpha)}\int_0^t(t-s)^{-\alpha}e^{-\eta(t-s)}\frac{d\omega}{ds}(s)ds.
\end{equation}
The fractional differentiation $D^{\alpha,\eta}$  is inverse operation of fractional integration that is defined by 
\begin{equation}\label{eq-1.8}
[I^{\alpha,\eta}\omega](t)=\int_0^t\frac{(t-s)^{\alpha-1}e^{-\eta(t-s)}}{\Gamma(\alpha)}\omega(s)ds.
\end{equation}
From equations \eqref{eq-1.7}-\eqref{eq-1.8}, we have 
\begin{equation}\label{p12g}
[D^{\alpha,\eta}\omega]=I^{1-\alpha,\eta}D\omega.
\end{equation}
The fractional derivatives are nonlocal and involve singular and non-integrable kernels ($t^{-\alpha}$, $0<\alpha<1$). We refer the readers to \cite{samkokilbasmarichev:93} and the rich references therein for mathematical description of the fractional derivative. 
The  fractional order or, in general, of convolution type are not only important from the theoretical point of view but also for applications. They naturally arise in physical, chemical, biological, ecological phenomena see for example \cite{parkKang:11},  and the rich references therein. They are used to describe memory and hereditary properties of various materials and processes. For example, in viscoelasticity, due to the nature of the material microstructure, both elastic solid and viscous fluid like response qualities are involved. Using Boltzmann assumption, we end up with a stress-strain relationship defined by a time convolution. Viscoelastic response occurs in a variety of materials, such as soils, concrete, rubber, cartilage, biological tissue, glasses, and polymers (see  \cite{bagleytorvik2:83,bagleytorvik3:83,bagleytorvik1:83} and \cite{mainardibonetti}). In our case, the fractional dissipation describes an active boundary viscoelastic damper designed for the purpose of reducing the vibrations (see \cite{Mbodje:06,mbomon:95}).
\subsection{Motivation} Since the work of \cite{Lions88-1}, the study of the stabilization of damped wave equations retains the attention of many authors sees, for instance, \cite{Komornik94, Bardos:92, Lebeau1996, Zuazua90, Burq:98, conrad:91}. Let us recall the scalar fractional damped wave equation, that is 
\begin{equation}\label{eq-1.9}
\left\{\begin{array}{ll}
u_{tt}(x,t)-u_{xx}(x,t)=0,&(x,t)\in (0,1)\times \mathbb{R}_+^\star,\\ \noalign{\medskip}
u(0,t)=0,&t\in  \mathbb{R}_+^\star,\\ \noalign{\medskip}
u_x(1,t)+\gamma\partial_t^{\alpha,\eta}u(1,t)=0,&t\in  \mathbb{R}_+^\star,
\end{array}
\right.
\end{equation}
where $\gamma >0$, $\eta \geq 0$ and $\alpha\in ]0,1[$. In \cite{Mbodje:06}, it was proved that the energy of system \eqref{eq-1.9} does not decay uniformly (exponentially) to zero but polynomial energy decay rate of type ${t^{-1}}$ is obtained. This result has been recently improved by Akil and Wehbe  \cite{akilwehbe01}, where an improved polynomial decay rate of the energy of the multi-dimensional case of system \eqref{eq-1.9} in a bounded domain $\Omega\subset \R^d$ with smooth boundary $\Gamma$ has been established. Roughly speaking, the authors proved that the energy of smooth solutions converges to zero as $t$ goes to infinity, as ${t^{-\frac{1}{1-\alpha}}}$. \\[0.1in]
The question we are interested in this paper is what are the stability properties of our wave-wave coupled system \eqref{eq-1.1}-\eqref{eq-1.4}. Indeed, this system involves two wave equations coupled by velocities with only one fractional damping acting on a part of the boundary of the first equation. The second equation is indirectly damped through the coupling between the two equations (see the Literature below for the history of this kind of damping). So, from the mathematical point of view, it is important to study the stability of a system coupling a polynomially stable wave equation with a conservative one. Moreover, the study of this kind of systems is also motivated by several physical considerations. Indirect damping of reversible systems occurs in many applications in engineering and mechanics (see the literature below). It is well known that when the boundary damping is static; i.e., when 
\begin{equation*}
u_x(1,t)+\gamma \partial_t u(1,t)=0,\ \ t\in \mathbb{R}^+
\end{equation*}
the energy of the solution decays exponentially under the conditions that $b$ is outside a well determined discrete set of exceptional values, $a=1$ and $b\not=k\pi$, for some $k\in\Z$ and polynomially under some arithmetic conditions on $a$ and $b$ (see \cite{Najdi-Thesis}).  As mentioned above, the presence of the fractional time derivative at the boundary has a great impact on the stabilization of the system and its interesting from the theoretical point of view and also in several applications. The stability of a system of wave equations coupled by velocities with only one fractional damping remains an interesting open problem.
\subsection{The main goal of this paper} In this paper, we provide a complete analysis for the stability of system \eqref{eq-1.1}-\eqref{eq-1.4}. Unlike the static damping case, the resolvent of the operator associated with system \eqref{eq-1.1}-\eqref{eq-1.4} is not compact. First, using  general criteria of Arendt and Batty \cite{arendt:88}, we show that our system is strongly stable if and only if the coupling parameter $b$ is outside a well determined discrete set $D_{a,b}$ of exceptional values.  Next, using a spectral analysis, we prove that our system is not uniformly (exponentially) stable even when   $b\notin D_{a,b}$. Consequently, we look for a polynomial energy decay rate for smooth initial data by applying a frequency domain approach combining with a multiplier technique. Indeed, we show that, for $b\notin D_{a,b}$, the energy decay rate of system \eqref{eq-1.1}-\eqref{eq-1.4}  is greatly influenced by the order of the fractional derivative $\alpha$, the algebraic nature of the coupling parameter $b$ (an additional condition on $b$) and by the arithmetic property of the ratio of the wave propagation speeds $a$. Indeed, under the equal speed propagation condition {\it i.e} $a=1$, we establish an optimal polynomial energy decay rate of type $t^{-\frac{2}{{1-\alpha}}}$ if the coupling parameter $b\notin \pi \Z$ and of type $t^{-\frac{2}{{5-\alpha}}}$ if the coupling parameter $b\in \pi \Z$. Furthermore, when the wave propagate with different  speeds {\it i.e} $a\not=1$, we prove that, for any rational number $\sqrt{a}$ and almost all irrational number $\sqrt{a}$, the energy of our system decays  polynomially to zero like as $t^{-\frac{2}{{5-\alpha}}}$. This result still be holds if $a\in \Q$, $\sqrt{a}\notin \Q$ and $b$ small enough.
\subsection{Literature} The notion of indirect damping mechanisms has been introduced by Russell in \cite{Russell01}, and since this time, it retains the attention of many authors. In particular, the fact that only one equation of the coupled system is damped refers to the so-called class of "indirect" stabilization problems initiated and studied in \cite{alabau1999,alabauCannarsaKomornik2002,Boussouira01}, and further studied by many authors, see for instance  \cite{ZuazuazZhang01,LiuRao01,Boussouira04} and the rich references therein. In \cite{alabau1999,Boussouira01}, Alabau studied the boundary indirect stabilization of a system of two second order evolution equations coupled through the zero order terms. The lack of uniform stability was proved by a compact perturbation argument and a polynomial energy decay rate of type $t^{-1/2}$ is obtained by a general integral inequality in the case where the waves propagate at the same speed and $\Omega$ is a star-shaped domain in $\R^N$, or in the case where the ratio of the wave propagation speeds of the two equations is equal $k^{-2}$ with $k$ being an integer and $\Omega$, is a cubic domain of $\R^3$. Under  appropriate geometric conditions, these results have been generalized by Alabau and L\'eautaud in \cite{Boussouira04}, to the cases for which the coupling parameter is a non negative function can vanish in some part of $\Omega$. The polynomial decay rate and the control of a $1-d$ hyperbolic-parabolic coupled system has been studied by Zhang and Zuazua in \cite{ZuazuazZhang01}. In \cite{LiuRao01}, Liu and Rao considered a system of two coupled wave equations with one boundary damping and they proved that the energy of the system decays at the rate $t^{-1}$ for smooth initial data on a $N$-dimensional domain $\Omega$ with usual geometrical condition when the waves propagate at the same speed. On the contrary, under some arithmetic condition on the ratio of the wave propagation speeds of the two equations, they established a polynomial energy decay rate for smooth initial data on a one-dimensional domain. In \cite{KodjaBader01}, Ammar-Khodja and Bader studied the simultaneous boundary stabilization of a system of two wave equations coupling through the velocity terms. In addition to the previously cited papers, we rapidly recall some previous studies done on the coupled systems with different kinds of damping. Ammari and Mehrenberger in \cite{AmmariMehre01}, gave a characterization of the stability of a system of two evolution equations coupling through the velocity terms subject to one bounded viscous feedback damping. Note nevertheless that the above system does not enter in the framework of this paper. In \cite{AlabauCannarsaGuglielmi2011}, Alabau, Cannarsa, and Guglielmi studied the indirect stabilization of weakly coupled systems with hybrid boundary conditions, they proved polynomial decay for the energy of solutions. Alabau, Wang, and Yu in \cite{AlabauZhiqiangLixin2017} studied the stability of the hyperbolic system of a wave-wave type which is coupled through the velocities with only one equation is directly damped by a nonlinear damping. They established an explicit energy decay formula in terms of the behavior of the nonlinear feedback close to the origin. Last but not least, we refer to \cite{alabau1999,alabauCannarsaKomornik2002,alabau2005,alabau2004,BassamWehbe2015,BassamWehbe2016,huangfFlu,wehbeyoussef02,kliu97,SoufWehbe03,wehbeyoussef01,NadineAli:Bresse,zhda:14,benaissa:17,pruss84,kapitonov96, Loreti-Rao:06,Khokerbsoufyane07} for the indirect stabilization and the indirect exact controllability of distributed systems with different kinds of damping. \\
The indirect control and stabilization of reversible systems occur in many applications in engineering and mechanics we quote \cite{SoufWehbe03, wehbeyoussef01, BassamWehbe2015, BassamWehbe2016, Alabau-Timo} for the Timoshenko system in bounded or unbounded domains, \cite{LiuRao-Bresse, wehbeyoussef02, WehbeNoun-Bresse, Alabau-Bresse, NadineAli:Bresse}  for the Bresse system.
\subsection{Description of the paper}\noindent This paper is organized as follows: In Section \ref{Section-2}, first, we show that the system \eqref{eq-1.1}-\eqref{eq-1.6} can be replaced by an augmented model by coupling the wave equation with a suitable diffusion equation for can reformulate into an evolution equation and we deduce the well-posedness property of the problem by the semigroup approach. Second, using a criteria  of Arendt-Batty we show that the augmented model is strongly stable in absence of compactness of the resolvent under a condition on $b$. In Section \ref{Section-3}, we show that the augmented model is non uniformly stable; i.e., (non exponential), this result is due to the fact that a subsequence of eigenvalues is due to the imaginary axis. In Section \ref{Section-4}, we show the polynomial energy decay rate for the system \eqref{eq-1.1}-\eqref{eq-1.6}.  Roughly speaking, we show that the energy of smooth solutions converges to zero as $t$ goes to infinity, as $t^{-s\left(\alpha\right)}$, where
\begin{equation*}
s\left(\alpha\right)=\left\{\begin{array}{lll}
\displaystyle{\frac{2}{1-\alpha}}& \displaystyle{\text{if }a=1 \text{ and }b\neq k\pi },
\\ \noalign{\medskip}
\displaystyle{\frac{2}{5-\alpha}}& \displaystyle{\text{if }a=1 \text{ and }b= k\pi},
\\ \noalign{\medskip}
\displaystyle{\frac{2}{5-\alpha}}& \displaystyle{\text{if }a\neq1,\ a\in \mathbb{Q},\ \sqrt{a}\not\in \mathbb{Q}  \text{ and } b\text{ small enough} },
\\ \noalign{\medskip}
\displaystyle{\frac{2}{5-\alpha}}& \displaystyle{\text{if }a\neq1,\  \sqrt{a}\in \mathbb{Q}, }
\\ \noalign{\medskip}
\displaystyle{\frac{2}{5-\alpha}}& \displaystyle{\text{if }a\neq1\ \text{and for almost }  \sqrt{a}\in \mathbb{R}/\mathbb{Q}. }
\end{array}\right.
\end{equation*}
\section{Well-Posedness and Strong Stability}\label{Section-2}
\noindent In this Section, we will study the strong stability of system \eqref{eq-1.1}-\eqref{eq-1.6} in the absence of the compactness of the resolvent. First, we will study the existence, uniqueness, and regularity of the solution of our system. 
\subsection{Augmented model and Well-Posedness}\label{Section-2.1}
\noindent Firstly, we reformulate system \eqref{eq-1.1}-\eqref{eq-1.6} into an augmented system. For this aim, we recall Theorem 2 in \cite{Mbodje:06}.
\begin{theoreme}\label{Theorem-2.1}$\left(\textbf{See Theorem 2 in  \cite{Mbodje:06}}\right)$
\noindent Let $\mu$ be the function defined by 
$$
\mu(\xi)=|\xi|^{\frac{2\alpha-1}{2}},\quad \xi\in \R\ \ {\rm \text{and}}\ \ \alpha\in ]0,1[.
$$
Then, the relation between the "input" $U$ and the "output" $O$ of the following system
$$ 
\begin{array}{lllll}
\partial_t\omega(\xi,t)+\left(\xi^2+\eta\right)\omega(\xi,t)-U(t)\mu(\xi)&=&0,&(\xi,t)\in  \R\times \R^+,\\ \noalign{\medskip}
\omega(\xi,0)&=&0,& \xi\in  \R,\\ \noalign{\medskip}
\displaystyle
O(t)-\frac{\sin(\alpha\pi)}{\pi}\int_{\R}\mu(\xi)\omega(\xi,t)d\xi&=&0,&\xi\in  \R,
\end{array}
$$
is given by 
$$
O=I^{1-\alpha,\eta}U,
$$
where $I^{\alpha,\eta}$ is given by \eqref{eq-1.8}.
\end{theoreme}
\noindent From Theorem \ref{Theorem-2.1} and equation \eqref{p12g}, system \eqref{eq-1.1}-\eqref{eq-1.6} may be recast into the following augmented model 
\begin{eqnarray}
u_{tt}-u_{xx}+by_t&=&0,\quad (x,t)\in (0,1)\times \R^+,\label{eq-2.1}\\ \noalign{\medskip}
y_{tt}-ay_{xx}-bu_t&=&0,\quad (x,t)\in  (0,1)\times \R^+,\label{eq-2.2}\\ \noalign{\medskip}
\omega_t(\xi,t)+(\xi^2+\eta)\omega(\xi,t)-u_t(1,t)\mu(\xi)&=&0,\quad  (\xi,t)\in  \R\times \R^+,\label{eq-2.3}\\ \noalign{\medskip}
y(0,t)=y(1,t)=u(0,t)&=&0,\quad t\in\R^+,\label{eq-2.4}\\
u_x(1,t)+\gamma\kappa(\alpha)\int_{\R}\mu(\xi)\omega(\xi,t)d\xi&=&0,\quad t\in\R^+,\label{eq-2.5}
\end{eqnarray}
where $\kappa(\alpha)=\frac{\sin(\alpha\pi)}{\pi}$, since $\alpha\in]0,1[$, then $\kappa(\alpha)>0.$  System \eqref{eq-2.1}-\eqref{eq-2.5}  is considered with the following initial conditions 
\begin{eqnarray}
u(x,0)=u_0(x),& u_t(x,0)=u_1(x),& x\in (0,1),\label{eq-2.6}\\ \noalign{\medskip}
y(x,0)=y_0(x),& y_t(x,0)=y_1(x),&  x\in (0,1),\label{eq-2.7}\\ \noalign{\medskip}
&\omega(\xi,0)=0,& \xi\in \mathbb{R}.\label{eq-2.8}
\end{eqnarray} 
Let us define the energy space
	\begin{equation*}
	\mathcal{H}=H_{L}^1(0,1)\times L^2(0,1)\times H_0^1(0,1) \times L^2(0,1)\times L^2(\R)
	\end{equation*}
	equipped with the following inner product
	$$
	\left<(u,v,y,z,\omega),(\tilde{u},\tilde{v},\tilde{y},\tilde{z},\tilde{\omega})\right>_\mathcal{H}=\int_{0}^1\left(v\bar{\tilde{v}}+u_x\bar{\tilde{u_x}}+z\bar{\tilde{z}}+ay_x\bar{\tilde{y_x}}\right)dx+\gamma\kappa(\alpha)\int_{\R}\omega(\xi)\bar{\tilde{\omega}}(\xi)d\xi,
	$$
	where $H_{L}^1(0,1)$ is given  by 
	$$
	H_{L}^1(0,1)=\left\{u\in H^1(0,1),\quad u(0)=0\right\}.
	$$
The energy of system \eqref{eq-2.1}-\eqref{eq-2.8} is given by
	\begin{equation*}
	E(t)=\frac{1}{2}\|(u,u_t,y,y_t,\omega)\|_\mathcal{H}^2.
	\end{equation*}
	For smooth solution, a direct computation gives
	\begin{equation*}
	E^\prime(t)=-\gamma \kappa(\alpha) \int_{\R}(\xi^2+\eta)|w(\xi,t)|^2d\xi\leq 0.
	\end{equation*}
	Then, system \eqref{eq-2.1}-\eqref{eq-2.8} is dissipative in the sense that its energy is a non-increasing function of the time variable $t$.
	Next, we define the unbounded linear operator $\mathcal{A}$ by 
\begin{equation*}
D(A)=\left\{\begin{array}{l}
\vspace{0.15cm}\displaystyle
U=(u,v,y,z,\omega) \in\mathcal{H};\ u\in H^2\left(0,1\right)\cap H_L^1(0,1),\ y\in H^2\left(0,1\right)\cap H_0^1(0,1)\\ 
\\ \noalign{\medskip}\displaystyle
v\in H_L^1(0,1),\ z\in H_0^1\left(0,1\right),\ -\left(\xi^2+\eta\right)\omega+v\left(1\right)\mu(\xi)\in L^2\left(\R\right),\\ 
\\ \noalign{\medskip}
\vspace{0.15cm}\displaystyle
u_x(1)+\gamma\kappa(\alpha) \int_{\mathbb{R}}\mu(\xi)\omega(\xi)d\xi=0,\ |\xi|\omega\in L^2\left(\R\right)
\end{array}\right\}
\end{equation*}
and 
$$
\mathcal{A}\left(u, v,y, z, \omega\right)=\left(
v,u_{xx}-bz, z, ay_{xx}+bv, -\left(\xi^2+\eta\right)\omega+v(1)\mu(\xi)\right).
$$
If $U=(u,u_t,y,y_t,\omega)$ is a regular solution of system \eqref{eq-2.1}-\eqref{eq-2.8}, then we rewrite this system as the following evolution equation
\begin{equation}\label{eq-2.9}
U_t=\mathcal{A}U,\quad
U(0)=U_0,
\end{equation}
where $U_0=(u_0,u_1,y_0,y_1,\omega)$.		
			\begin{pro}\label{Theorem-2.2}
\noindent The unbounded linear operator $\mathcal{A}$ is m-dissipative in the energy space $\mathcal{H}$.
		\end{pro}
		\begin{proof}
			For all $U=(u,v,y,z,\omega)\in D\left(\mathcal{A}\right)$, we have 
			\begin{equation*}
			\Re\left(\left<\mathcal{A}U,U\right>_{\mathcal{H}}\right)=-\gamma\kappa(\alpha) \int_{\mathbb{R}}(\xi^2+\eta)|\omega(\xi)|^2d\xi\leq 0,
			\end{equation*}
			which implies that $\mathcal{A}$ is dissipative. Now, let $F=(f_1,f_2,f_3,f_4,f_5)$, we prove the existence of $U=(u,v,y,z,\omega)\in D(\mathcal{A})$, solution of the equation 
			\begin{equation}\label{eq-2.10}
			(I-\mathcal{A})U=F.
			\end{equation}
			Equivalently, we have the following system 
			\begin{eqnarray}
			u-v&=&f_1,\label{eq-2.11}\\
			v-u_{xx}+bz&=&f_2,\label{eq-2.12}\\
			y-z&=&f_3,\label{eq-2.13}\\
			z-ay_{xx}-bv&=&f_4,\label{eq-2.14}\\
			(1+\xi^2+\eta)\omega(\xi)-v(1)\mu(\xi)&=&f_5(\xi).\label{eq-2.15}
			\end{eqnarray}
			Using equations \eqref{eq-2.15}, \eqref{eq-2.11} and the fact that $\eta \geq 0$, we get 
			\begin{equation}\label{eq-2.16}
			\omega(\xi)=\frac{f_5(\xi)}{1+\xi^2+\eta}+\frac{u(1)\mu(\xi)}{1+\xi^2+\eta}-\frac{f_1(1)\mu(\xi)}{1+\xi^2+\eta}.
			\end{equation}
			Inserting equations \eqref{eq-2.11} and  \eqref{eq-2.13} into \eqref{eq-2.12} and \eqref{eq-2.14}, we get 
			\begin{eqnarray}
			u-u_{xx}+by&=&f_1+f_2+bf_3,\label{eq-2.17}\\
			y-ay_{xx}-bu&=&-bf_1+f_3+f_4,\label{eq-2.18}
			\end{eqnarray}
			with the boundary conditions 
			\begin{equation}\label{eq-2.19}
			u(0)=0,\quad u_x(1)+M_2(\eta,\alpha)u(1)=M_2(\eta,\alpha)f_1(1)-M_1(\eta,\alpha)\quad \text{and}\quad y(0)=y(1)=0,
			\end{equation}
where 
$$M_1(\eta,\alpha)= \gamma\kappa(\alpha) \int_{\mathbb{R}}\frac{\mu(\xi)f_5(\xi)}{1+\eta+\xi^2}d\xi \quad \text{and} \quad M_2(\eta,\alpha)=\gamma\kappa(\alpha) \int_{\mathbb{R}}\frac{\mu^2(\xi)}{1+\eta+\xi^2}d\xi.$$
Using the fact $f_5\in L^2(\R)$, the definition of $\mu(\xi)$ and the fact that $\alpha\in ]0,1[,\ \eta\geq0$, it is easy to check that $M_1(\eta,\alpha)$ and $M_2(\eta,\alpha)$ are well defined. 
So, using Lax-Milligram Theorem and the fact that the second members of \eqref{eq-2.17} and \eqref{eq-2.18} are in $L^2(0,1)$, we deduce that there exists $(u,y)\in H^2(0,1)\cap H_L^1(0,1)\times H^2(0,1)\cap H_0^1(0,1)$ unique strong solution of \eqref{eq-2.17}-\eqref{eq-2.19}. Consequently, defining $v=u-f_1$, $z=y-f_3$ and $\omega$ by \eqref{eq-2.16}, we deduce that $U=(u,v,y,z,\omega)\in D(\AA)$ is the unique solution of \eqref{eq-2.10}. The proof is thus complete.
\end{proof}
\noindent From Proposition \ref{Theorem-2.2}, the operator $\AA$ is m-dissipative on $\HH$ and consequently, generates a $C_0-$semigroup of contractions $e^{t\AA}$ following Lummer-Phillips Theorem (see in \cite{liu:99} and \cite{pazy}). Then the solution of the evolution equation \eqref{eq-2.9}	admits the following representation
$$
U(t)=e^{t\AA}U_0,\quad t\geq 0,
$$
which leads to the well-posedness of \eqref{eq-2.9}. Hence, we have the following result.
		\begin{theoreme}\label{Theorem-2.3}
Let $U_0\in \mathcal{H}$ then, problem \eqref{eq-2.9} admits a unique weak solution $U$ satisfies 
		$$U(t)\in C^0\left(\R^+,\mathcal{H}\right).$$
Moreover, if $U_0\in D(\mathcal{A})$ then, problem \eqref{eq-2.9} admits a unique strong solution $U$ satisfies 
$$U(t)\in C^1\left(\R^+,\mathcal{H}\right)\cap C^0(\R^+,D(\mathcal{A})).$$
	\end{theoreme}
\subsection{Strong Stability}\label{Section-2.2}
\noindent In this part, we study the strong stability of system \eqref{eq-2.1}-\eqref{eq-2.8} in the sense that its energy converges to zero when $t$ goes infinity for all initial data in $\HH$. Since the resolvent of $\AA$ is not compact, then the classical methods such as Lasalle's invariance principle \cite{Slemrod:89} or the spectrum decomposition theory of Benchimol \cite{benchimol:78} are not applicable in this case. Then, we will use a general criteria of Arendt-Batty \cite{arendt:88}, following which a $C_0-$semigroup of contractions $e^{t\AA}$ in a Banach space is strongly stable, if $\AA$ has no pure imaginary eigenvalues and $\sigma(\AA)\cap i\R$ contains only a countable number of elements. So, we will prove the following stability result.
\begin{theoreme}\label{Theorem-2.4}
Assume that $\eta\geq 0$, then, the $C_0-$semigroup of contractions $e^{t\AA}$ is strongly stable on $\HH$ in the sense that $\displaystyle{\lim_{t\to +\infty}\|e^{t\AA}U_0\|_{\HH}=0}$ for all $U_0\in \HH$ if and only if 
\begin{equation}\tag{${\rm SC}$}
b^2\neq \frac{(k_1^2-ak_2^2)(ak_1^2-k_2^2)}{(a+1)(k_1^2+k_2^2)}\pi^2,\quad \forall k_1,k_2\in \Z.
\end{equation}
\end{theoreme}
\noindent For the proof of Theorem \ref{Theorem-2.4}, we need the following Lemmas.
\begin{lemma}\label{Theorem-2.5}
Assume that $\eta \geq 0$ and $b$ satisfying condition $(\rm SC)$. Then, for all $\lambda\in \R$, we have 
$$
\ker\left(i\la I-\AA\right)=\{0\}.
$$ 
\end{lemma}
\begin{proof}
Let $U\in D(\AA)$ and let $\la \in \R$, such that 
\begin{equation*}
\AA U=i\la U.
\end{equation*}
Equivalently, we have 
\begin{eqnarray}
v&=&i\la u,\label{eq-2.20}\\
u_{xx}-bz&=&i\la v,\label{eq-2.21}\\
z&=&i\la y,\label{eq-2.22}\\
ay_{xx}+bv&=&i\la z,\label{eq-2.23}\\
-(\xi^2+\eta)\omega(\xi)+v(1)\mu(\xi)&=&i\la \omega(\xi)\label{eq-2.24}.
\end{eqnarray}
Next, a straightforward computation gives 
\begin{equation*}
0=\Re\left<i\lambda U,U\right>_{\HH}=\Re\left<\AA U,U\right>_{\HH}=-\gamma\kappa(\alpha)\int_{\R}(\xi^2+\eta)|\omega(\xi)|^2d\xi.
\end{equation*}
Then,  we deduce that 
\begin{equation}\label{eq-2.25}
\omega=0\quad \text{a.e.\ in}\quad \R.
\end{equation}
It follows, from equations \eqref{eq-2.5}, \eqref{eq-2.24} and \eqref{eq-2.25}, that 
\begin{equation}\label{eq-2.26}
u_x(1)=0\quad \text{and}\quad v(1)=0.
\end{equation}
Substituting equations \eqref{eq-2.20}, \eqref{eq-2.22} in equations \eqref{eq-2.21}, \eqref{eq-2.23} and using equation \eqref{eq-2.26}, we get
\begin{eqnarray}
\la^2u+u_{xx}-i\la by&=&0,\label{eq-2.27}\\
\la^2y+ay_{xx}+i\la bu&=&0,\label{eq-2.28}\\
u(0)=\la u(1)=u_x(1)=y(0)=y(1)&=&0. \label{eq-2.29}
\end{eqnarray}
Like as \cite{Najdi-Thesis}, we distinguish two cases:\\[0.1in]
\textbf{Case 1}. If $\la =0$, by direct calculation, we deduce that $u=0,\ y=0$, and consequently $U=0$.\\[0.1in]
\textbf{Case 2}. If $\la \neq 0$, combining equations \eqref{eq-2.27}-\eqref{eq-2.29}, we get the following systems 
\begin{eqnarray}
au_{xxxx}+(a+1)\la^2 u_{xx}+\la^2(\la^2-b^2)u&=&0,\label{eq-2.30}\\
u(0)=u_{xx}(0)&=&0,\label{eq-2.31}\\
u(1)=u_{xx}(1)&=&0,\label{eq-2.32}\\
u_x(1)&=&0.\label{eq-2.33}
\end{eqnarray}
The solution $u$ of equation \eqref{eq-2.30} is given by $\displaystyle{u(x)=\sum_{j=1}^{4}c_je^{r_jx}}$, where $c_j\in \C$ and 
\begin{equation}\label{eq-2.34}
r_1=\sqrt{\frac{-\la^2(a+1)-\la\sqrt{\la^2(a-1)^2+4ab^2}}{2a}},\quad r_2=-r_1,
\end{equation}
\begin{equation}\label{eq-2.35}
r_3=\sqrt{\frac{-\la^2(a+1)+\la\sqrt{\la^2(a-1)^2+4ab^2}}{2a}},\quad r_4=-r_3.
\end{equation}
If $\la=\pm b$, then system \eqref{eq-2.30}-\eqref{eq-2.33} admits only the trivial solution, in this case the proof is complete.  Otherwise, if $\la \neq \pm b$, since $r_1^2-r_3^2\neq 0$, then using boundary conditions \eqref{eq-2.31}, we get 
$$
u(x)=2c_1\sinh(r_1x)+2c_3\sinh(r_3x).
$$ 
From the boundary conditions \eqref{eq-2.32}, we distinguish the following four cases:
\begin{enumerate}
\item[1.] $\sinh(r_1)\neq 0$ and $\sinh(r_3)\neq 0$. Since $r_1^2-r_3^2\neq 0$, then using boundary conditions \eqref{eq-2.32}, it is easy to see that $u(x)=0$. Consequently $U=0$.\\
\item[2.]  If  $\sinh(r_1)\neq0$ and $\sinh(r_3)=0$. Using boundary conditions \eqref{eq-2.32} we deduce $c_1=0$ and then $u(x)=2c_3\sinh(r_3x)$. Finally, using boundary condition \eqref{eq-2.33}, we get $u(x)=0$. Consequently $U=0$.\\
\item[3.]  If $\sinh(r_1)=0$ and $\sinh(r_3)\neq 0$. By a same argument as in previous case, we get $u(x)=0$ and consequently $U=0$.\\
\item[4.] If $\sinh(r_1)=0$ and $\sinh(r_3)=0$. It follows that 
\begin{equation}\label{eq-2.36}
r_1=i k_1\pi\quad \text{and}\quad r_3=i k_2\pi,\quad \text{where}\  k_1, k_2\in \Z.
\end{equation}
Inserting equation \eqref{eq-2.36} into equations \eqref{eq-2.34} and \eqref{eq-2.35} respectively, we get
\begin{equation}\label{eq-2.37}
\la^2(a+1)+\la\sqrt{\la^2(a-1)^2+4ab^2}=2a k_1^2\pi^2
\end{equation}
and 
\begin{equation}\label{eq-2.38}
\la^2(a+1)-\la\sqrt{\la^2(a-1)^2+4ab^2}=2a k_2^2\pi^2.
\end{equation}
By adding equations \eqref{eq-2.37} and \eqref{eq-2.38}, we obtain 
\begin{equation}\label{eq-2.39}
\la^2=\frac{a}{a+1}\left( k_1^2+ k_2^2\right)\pi^2.
\end{equation}
Subtracting   \eqref{eq-2.37} from \eqref{eq-2.38}, we obtain
\begin{equation}\label{eq-2.40}
\la\sqrt{\la^2(a-1)^2+4ab^2}=a\left( k_1^2- k_2^2\right)\pi^2.
\end{equation}
Inserting \eqref{eq-2.40} in \eqref{eq-2.39}, we get 
\begin{equation}\label{eq-2.41}
b^2=\frac{\left( k_1^2-a k_2^2\right)\left(a k_1^2- k_2^2\right)}{(a+1)\left( k_1^2+ k_2^2\right)}\pi^2.
\end{equation}
This contradicts ${\rm (SC)}$.
\end{enumerate}
Conversely, if \eqref{eq-2.41} holds, then $i\la$ $\left(\text{where }\la\text{ is given in \eqref{eq-2.39}}\right)$ is an eigenvalue of $\AA$ with the corresponding eigenvector $$U=\left(u,i\la u,-\frac{i}{\la b}\left(\la^2u+u_{xx}\right),\frac{1}{b}\left(\la^2 u+u_{xx}\right),0\right),$$ such that
$$
u(x)=2i\frac{ k_2\left(a k_1^2- k_2^2\right)}{ k_1\left( k_1^2-a k_2^2\right)}\sin\left( k_1\pi x\right)+2i\sin\left( k_2 \pi x\right).
$$
Consequently, if \eqref{eq-2.41} does not hold, then $i\la$ is not an eigenvalue of $\AA$. The proof is thus complete.
\end{proof}
\begin{lemma}\label{Theorem-2.6}
Assume that $\eta=0$. Then, the operator $-\AA$ is not invertible and consequently $0\in \sigma(\AA)$.
\end{lemma}
\begin{proof}
Set $F=\left(\sin(x),0,0,0,0\right)\in \HH$ and assume that there exists $U=(u,v,y,z,\omega)\in D(\AA)$ such that 
$$
-\AA U= F.
$$ 
It follows that 
\begin{equation}\label{eq-2.42}
v=-\sin(x)\ \ \text{in}\ \ (0,1)\ \ \ \text{and}\ \ \ \xi^2\omega+\sin(1)\mu(\xi)=0.
\end{equation}
From equation \eqref{eq-2.42}, we deduce that $\omega(\xi)=\xi^{\frac{2\alpha-5}{2}}\sin(1)\notin L^2(\R)$, therefore the assumption of the existence of $U$ is false and consequently, the operator $-\AA$ is not invertible.  The proof is thus complete.
\end{proof}
\begin{lemma}\label{Theorem-2.7}
Assume $b$ satisfies condition ${\rm (SC)}$ and  assume that {\rm(}$\eta>0$ and $\lambda\in \R${\rm)} or {\rm(}$\eta=0$ and $\lambda\in \R^{\ast}${\rm)}. Then, for any $f=(h,g)\in (L^2(0,1))^2$, the following problem 
	\begin{equation}\label{eq-2.43}
	\left\{\begin{array}{lllll}
	\la^2u+u_{xx}-i\la by&=&h,&x\in (0,1),
\\ \noalign{\medskip}
	\la^2y+ay_{xx}+i\la bu&=&g,&x\in (0,1),
\\ \noalign{\medskip}
	u(0)&=&0,
\\ \noalign{\medskip}
	y(0)=y(1)&=&0,
\\ \noalign{\medskip}
	u_x(1)+\left(\la^2c_1(\la,\eta,\alpha)+i\la c_2(\la,\eta,\alpha)\right)u(1)&=&0,
	\end{array}
	\right.
	\end{equation}
	admits a unique strong solution $(u,y)\in \left(H^2(0,1)\cap H_{L}^1(0,1)\right)\times \left(H^2(0,1)\cap H_0^1(0,1)\right)$, where 
	\begin{equation}\label{eq-2.44}
	c_1(\la,\eta,\alpha)=\gamma\kappa(\alpha)\int_{\R}\frac{\mu^2(\xi)}{\la^2+(\xi^2+\eta)^2}d\xi\quad \text{and}\quad c_2(\la,\eta,\alpha)=\gamma\kappa(\alpha)\int_{\R}\frac{\mu^2(\xi)(\xi^2+\eta)}{\la^2+(\xi^2+\eta)^2}d\xi.
	\end{equation}
\end{lemma}
\begin{rem}\label{Theorem-2.8}
Using the definition of $\mu$ and the fact that $\alpha\in ]0,1[$, $\eta\geq 0$, we deduce that the coefficients $c_1(\la,\eta,\alpha)$ and $c_2(\la,\eta,\alpha)$ are well defined.
\end{rem}
\begin{proof}
In the case $\eta>0$ and $\la=0$, by applying  Lax-Milligram Theorem, it is easy to see that  problem \eqref{eq-2.43} admits a unique strong solution $(u,y)\in \left(H^2(0,1)\cap H_{L}^1(0,1)\right)\times \left(H^2(0,1)\cap H_0^1(0,1)\right)$. So, we study the existence of strong solution of \eqref{eq-2.43} in the case $\eta\geq 0\ \text{and}\ \la\in \R^{\ast}$.  To this aim, we define the following linear unbounded operator ${\mathcal{L}}$ by 
$$
D(\mathcal{L})=\left\{(u,y)\in \left(H^2(0,1)\cap H_L^1(0,1)\right)\times \left(H^2(0,1)\cap H_0^1(0,1)\right);\quad u_x(1)+\left(\la^2c_1+i\la c_2\right)u(1)=0\right\}
$$
and 
$$
\mathcal{L}\mathcal{U}=\left(-u_{xx}+i\la by, -ay_{xx}-i\la bu\right),\qquad \forall\    \mathcal{U}=(u,y)\in D({\mathcal{L}}).
$$
Let us consider the following problem 
	\begin{equation}\label{eq-2.45}
	\left\{\begin{array}{lllll}
	-u_{xx}+i\la by&=&h,&x\in(0,1),
\\ \noalign{\medskip}
	-ay_{xx}-i\la bu&=&g,&x\in(0,1),
\\ \noalign{\medskip}
	u(0)&=&0,
\\ \noalign{\medskip}
	y(0)=y(1)&=&0,
\\ \noalign{\medskip}
	u_x(1)+\left(\la^2c_1(\la,\eta,\alpha)+i\la c_2(\la,\eta,\alpha)\right)u(1)&=&0.
	\end{array}
	\right.
	\end{equation}
Tanks to Lax-Milligram Theorem, it is easy to see  that \eqref{eq-2.45} has a unique strong solution $(u,y)\in D({\mathcal{L}})$. In addition, we have 
\begin{equation*}
\|(u,y)\|_{H^2(0,1)\times H^2(0,1)}\leq c\|(h,g)\|_{L^2(0,1)\times L^2(0,1)}.
\end{equation*}
It follows, from the above inequality and the compactness of the embeddings $H_L^1(0,1)\times H_0^1(0,1)$ into $L^2(0,1)\times L^2(0,1)$, that the inverse operator $\mathcal{L}^{-1}$ is compact in $L^2(0,1)\times L^2(0,1)$. Then applying $\mathcal{L}^{-1}$ to \eqref{eq-2.43}, we get  
    \begin{equation*}
    	\left(\la^2\mathcal{L}^{-1}-I\right)\mathcal{U}=\mathcal{L}^{-1}f,\quad \text{where}\quad \mathcal{U}=(u,y)\ \text{and}\ f=(h,g).
    \end{equation*}
On the other hand, since $b$ satisfies condition ${\rm (SC)}$, by  same computation used in Lemma \ref{Theorem-2.5}, we show that $\ker\left(\la^2\mathcal{L}^{-1}-I\right)=\{0\}$. Then, following Fredholm's alternative,  system \eqref{eq-2.43} admits a unique solution.  The proof is thus complete.
\end{proof}
\begin{lemma}\label{Theorem-2.9}
If $\eta>0$, then for all $\la\in\R$, we have
$$
R(i\la I-\mathcal{A})=\mathcal{H}.
$$
If $\eta=0$, then for all $\la\in \R^{\ast}$, we have 
$$
R(i\la I-\mathcal{A})=\mathcal{H}.
$$
\end{lemma}
\begin{proof}
We distinguish two cases:  $\eta>0$ and $\eta=0$. Since the argument is entirely similar for the two cases, we only provide one of them. Assume that $\eta>0$ and let $\la\in \R$.  Set $F=(f_1,f_2,f_3,f_4,f_5) \in \mathcal{H}$, we look for $U=(u,v,y,z,\omega) \in D(\mathcal{A})$ solution of 
			\begin{equation}\label{eq-2.46}
			(i\la I-\mathcal{A})U=F.
			\end{equation}
			Equivalently, we have 
			\begin{equation*}
			\left\{
			\begin{array}{ll}
			
			\displaystyle{v=i\la u-f_1}, \\ \noalign{\medskip}
			\displaystyle{z=i\la y-f_3}, \\ \noalign{\medskip}
			\displaystyle{i\la v-u_{xx}+bz=f_2}, \\ \noalign{\medskip}
			\displaystyle{i\la z-ay_{xx}-bv=f_4}, \\ \noalign{\medskip}
			\displaystyle{(i\la+\xi^2+\eta)\omega-v(1)\mu(\xi)=f_5}, \\ \noalign{\medskip}
			\displaystyle{u(0)=y(0)=y(1)=0,\ u_x(1)+\gamma\kappa(\alpha)\int_{\R}\mu(\xi)\omega(\xi)d\xi=0}.
			
			\end{array}
			\right.
			\end{equation*}
			By eliminating $v$ and $\omega$ from the above system, we get the following system 
			\begin{equation}\label{eq-2.47}
			\left\{
			\begin{array}{ll}

			\displaystyle{\la^2u+u_{xx}-i\la by=-f_2-i\la f_1-bf_3}, \\ \noalign{\medskip}
			\displaystyle{\la^2y+ay_{xx}+i\la bu=-f_4-i\la f_3+bf_1}, \\ \noalign{\medskip}
			\displaystyle{u(0)=y(0)=y(1)=0},\\ \noalign{\medskip}
			\displaystyle{u_{x}(1)+(\la^2c_1(\la,\eta,\alpha)+i\la c_2(\la,\eta,\alpha))u(1)=-i\la c_1f_1(1)+c_2f_1(1)+I_1(\la,\eta,\alpha)+I_2(\la,\eta,\alpha)}.
			\end{array}
			\right.
			\end{equation}
			where $c_1(\la,\eta,\alpha)$ and $c_2(\la,\eta,\alpha)$ are defined in \eqref{eq-2.44} and $I_1(\la,\eta,\alpha),\ I_2(\la,\eta,\alpha)$ are given by 
			$$
			I_1(\la,\eta,\alpha)=i\la \gamma\kappa(\alpha)\int_{\R}\frac{f_5(\xi)\mu(\xi)}{\la^2+(\xi^2+\eta)^2}d\xi\quad \text{and}\quad I_2(\la,\eta,\alpha)=-\gamma\kappa(\alpha)\int_{\R}\frac{f_5(\xi)\mu(\xi)(\xi^2+\eta)}{\la^2+(\xi^2+\eta)^2}d\xi.
			$$
Using the definition of $\mu$ and the fact that $f_5\in L^2(\R)$, $\eta\geq0$ and $\alpha\in ]0,1[$, we deduce that the integrals $I_1(\la,\eta,\alpha)$ and $I_2(\la,\eta,\alpha)$ are well defined. Now, let $(\varphi,\psi)\in \left(H^2(0,1)\cap H_L^1(0,1)\right)\times \left(H^2(0,1)\cap H_0^1(0,1)\right)$ be the strong solution of  
			\begin{equation}\label{eq-2.48}
			\left\{\begin{array}{lllll}
		    -\varphi_{xx}+i\la b\psi&=&0,&x\in(0,1)
\\ \noalign{\medskip}
		    -\psi_{xx}-i\la b\varphi&=&0,&x\in(0,1)
\\ \noalign{\medskip}
		    \varphi(0)&=&0,
\\ \noalign{\medskip}
		    \psi(0)=\psi(1)&=&0,
\\ \noalign{\medskip}
		    \varphi_x(1)&=&I_1(\la,\eta,\alpha)+I_2(\la,\eta,\alpha).
			\end{array}
			\right.
			\end{equation}
Setting $\tilde{u}=u+\varphi$ and $\tilde{y}=y+\psi$. Then, from \eqref{eq-2.47} and \eqref{eq-2.48}, we get 
			\begin{equation}\label{eq-2.49}
			\left\{\begin{array}{lllll}
			\la^2\tilde{u}+\tilde{u}_{xx}-i\la b\tilde{y}&=&\la^2\varphi-f_2-i\la f_1-bf_3,\qquad x\in(0,1),
\\ \noalign{\medskip}
			\la^2\tilde{y}+a\tilde{y}_{xx}+i\la b\tilde{u}&=&\la^2\psi-f_4-i\la f_3+bf_1,\qquad x\in(0,1),
\\ \noalign{\medskip}
			\tilde{u}(0)&=&0,
\\ \noalign{\medskip}
			\tilde{y}(0)=\tilde{y}(1)&=&0,
\\ \noalign{\medskip}
			\tilde{u}_x(1)+(\la^2c_1(\la,\eta,\alpha)+i\la c_2(\la,\eta,\alpha))\tilde{u}(1)&=& \mathcal{Y}
			\end{array}
			\right.
			\end{equation}
			where
			$$\mathcal{Y}=-i\la c_1(\la,\eta,\alpha) f_1(1)+c_2(\la,\eta,\alpha) f_1(1)+(\la^2c_1(\la,\eta,\alpha)+i\la c_2(\la,\eta,\alpha))\varphi(1).$$
			Next, let $\theta\in H^2(0,1)\cap H_L^1(0,1)$, such that 
			$$
			\theta(1)=0,\quad \theta_x(1)=\mathcal{Y}.
			$$
Setting $\chi=\tilde{u}-\theta$. Then, from \eqref{eq-2.49}, we get
			\begin{equation}\label{eq-2.50}
			\left\{\begin{array}{lllll}
			\la^2\chi+\chi_{xx}-i\la b\tilde{y}&=&\la^2\varphi-\la^2\theta-\theta_{xx}-f_2-i\la f_1-bf_3,&x\in(0,1),
\\ \noalign{\medskip}
			\la^2\tilde{y}+a\tilde{y}_{xx}+i\la b\chi&=&\la^2\psi-i\la b\theta-f_4-i\la bf_3+bf_1,&x\in(0,1),
\\ \noalign{\medskip}
			\chi(0)&=&0,
\\ \noalign{\medskip}
			\tilde{y}(0)=\tilde{y}(1)&=&0,
\\ \noalign{\medskip}
			\chi_x(1)+\left(\la^2c_1+i\la c_2\right)\chi(1)&=&0.
			\end{array}
			\right.
			\end{equation}
			Using Lemma \ref{Theorem-2.7}, problem \eqref{eq-2.50} has a unique solution $(\chi,\tilde{y})\in \left(H^2(0,1)\cap H_L^1(0,1)\right)\times \left(H^2(0,1)\cap H_0^1(0,1)\right)$ and therefore problem \eqref{eq-2.47} has a unique solution $(u,y)\in \left(H^2(0,1)\cap H_L^1(0,1)\right)\times \left(H^2(0,1)\cap H_0^1(0,1)\right)$. So, defining $v=i\la u-f_1$, $z=i\la y-f_3$  and 
	        $$
		    \omega(\xi)=\frac{f_5(\xi)}{i\la +\xi^2+\eta}+\frac{i\la u(1)\mu(\xi)}{i\la+\xi^2+\eta}-\frac{f_1(1)\mu(\xi)}{i\la +\xi^2+\eta},
		    $$
we obtain $U=(u,v,y,z,\omega)\in D(\mathcal{A})$ solution of \eqref{eq-2.46}. The proof is thus complete.
\end{proof}
\noindent \textbf{Proof of Theorem \ref{Theorem-2.4}.} Following a general criteria of Arendt-Batty see \cite{arendt:88}, the $C_0-$semigroup of contractions $e^{t\AA}$ is strongly stable, if $\sigma\left(\AA\right)\cap i\R$ is countable and no eigenvalue of $\AA$ lies on the imaginary axis. First, using Lemma \ref{Theorem-2.5}, we directly deduce that $\AA$ has no pure imaginary eigenvalues. Next, using Lemmas \ref{Theorem-2.6}, \ref{Theorem-2.9} and with the help of the closed graph Theorem of Banach, we deduce that $\sigma(\AA)\cap i\R=\{\emptyset\}$ if $\eta>0$ and $\sigma(\AA)\cap i\R=\{0\}$ if $\eta=0$. Consequently, the $C_0-$semigroup $e^{t\AA}$ is strongly stable and the proof is thus complete.
\section{Non Uniform Stability}\label{Section-3}
In this section we show that uniform stability (i.e. exponential stability) does not hold even in the case $a=1$ and $b\notin \pi\Z$. This result is due to the fact that a sub-sequence of eigenvalues of $\AA$ is close to the imaginary axis.
\subsection{Non Uniform Stability in the case $\mathbf{a=1}$}\label{Section-3.1}
		\noindent In this part, we assume that $a=1$, $\eta >0$ and		\begin{equation}\tag{${\rm SC1 }$}
		b\neq \frac{\pi}{\sqrt{2}}\frac{k_1^2-k_2^2}{\sqrt{k_1^2+k_2^2}}.
		\end{equation}
		Our goal is to show that system \eqref{eq-2.1}-\eqref{eq-2.8} is not exponentially stable. This result is due to the fact that a subsequence of the eigenvalues of $\mathcal{A}$ is close to the imaginary axis.
\begin{theoreme}\label{Theorem-3.1}
			Assume that $a=1$ and that condition ${\rm (SC1)}$ holds. Then the semigroup of contractions $e^{t\AA}$ is not uniformly stable in the energy space $\HH$.
\end{theoreme}
\noindent For the proof of Theorem \ref{Theorem-3.1}, we aim to show that an infinite number of eigenvalues of $\mathcal{A}$ approach the imaginary axis. Since $\mathcal{A}$ is dissipative, we fix $\alpha_0>0$ large enough and we study the asymptotic behavior of the large eigenvalues $\la$ of 
$\mathcal{A}$ in the strip 
$$S=\left\{\lambda \in \C:-\alpha_0\leq \Re(\la)\leq 0\right\}.$$ 
For the aim, let $\la\in \mathbb{C}^*$ be an eigenvalue of $\mathcal{A}$ and let $U=(u,\lambda u,y,\lambda y,\omega)\in D(\mathcal{A})$ be an associated eigenvector such $\|U\|_{\mathcal{H}}=1$. Then, we have
		\begin{eqnarray}
		\la^2u-u_{xx}+b\la y=0,\label{eq-3.1}\\
		\la^2y-y_{xx}-b\la u=0,\label{eq-3.2}\\
		\omega(\xi)=\frac{ |\xi|^{\alpha-\frac{1}{2}}}{\la+\xi^2+\eta} \la u(1),\label{eq-3.3}
		\end{eqnarray}
with the boundary conditions		
\begin{eqnarray}
y(0)=y(1)=u(0)=0,\label{eq-3.4}
\\
u_x(1)=-\gamma\kappa(\alpha)\int_{\mathbb{R}}|\xi|^{\alpha-\frac{1}{2}}\omega(\xi)d\xi.\label{eq-3.5}				
\end{eqnarray}	
First, from	\eqref{eq-3.2}, we have
\begin{equation}
			u=\frac{1}{b\la}\left(\la^2y-y_{xx}\right),\label{eq-3.6}
\end{equation}
inserting \eqref{eq-3.6}  in \eqref{eq-3.1}, we get
			\begin{equation}\label{eq-3.7}
			y_{xxxx}-2\la^2y_{xx}+\la^2(\la^2+b^2)y=0.
			\end{equation}
Now, from \eqref{eq-3.1} and \eqref{eq-3.4}, we get
\begin{equation}
y(0)=y(1)=y_{xx}(0)=0.\label{eq-3.8}
\end{equation}
Next, inserting \eqref{eq-3.3} in \eqref{eq-3.5}, we get
\begin{equation}\label{eq-3.9}
u_x(1)=-\gamma\kappa(\alpha) \la u(1)\int_{\mathbb{R}}\frac{ |\xi|^{\alpha-1}}{\la+\xi^2+\eta}  d\xi.
\end{equation}
On the other hand, from \eqref{eq-3.6}, we have
\begin{equation}
			u(1)=\frac{1}{b\la}\left(\la^2y(1)-y_{xx}(1)\right)\ \ \ \text{and}\ \ \ u_x(1)=\frac{1}{b\la}\left(\la^2y_x(1)-y_{xxx}(1)\right)\label{eq-3.10},
\end{equation}
then inserting 	\eqref{eq-3.10} in \eqref{eq-3.9} and using the fact that 	$\kappa(\alpha)\int_{\mathbb{R}}\frac{ |\xi|^{\alpha-1}}{\la+\xi^2+\eta}  d\xi	=(\la+\eta)^{\alpha-1}$, we get 
			\begin{equation}\label{eq-3.11}
			y_{xxx}(1)+\gamma\la(\la+\eta)^{\alpha-1}y_{xx}(1)-\la^2y_x(1)=0.
			\end{equation} 
Therefore, from 	\eqref{eq-3.7}, 	\eqref{eq-3.8} and \eqref{eq-3.11}, we get
\begin{equation}\label{eq-3.12}
\left\{
\begin{array}{ll}
\displaystyle{y_{xxxx}-2\la^2y_{xx}+\la^2(\la^2+b^2)y=0,}
\\[0.1in]
\displaystyle{y(0)=y_{xx}(0)=0,}
\\[0.1in]
\displaystyle{y(1)=0,\ y_{xxx}(1)+\gamma\la(\la+\eta)^{\alpha-1}y_{xx}(1)-\la^2y_x(1)=0.}
\end{array}
\right.
\end{equation}
The general solution of equation \eqref{eq-3.12} is given by 
			\begin{equation*}
			y(x)=c_1 \sinh(r_1 x)+c_2\sinh(r_2 x)+c_3\cosh(r_1 x)+c_4\cosh(r_2 x),
			\end{equation*}
			where $c_1,\ c_2,\ c_3,\ c_4\in\mathbb{C}$, and
\begin{equation}\label{eq-3.13}
r_1=\lambda \left(1+\frac{ib}{\la}\right)^{\frac{1}{2}},\qquad r_2=\lambda \left(1-\frac{ib}{\la}\right)^{\frac{1}{2}}.
\end{equation}
From the boundary conditions in \eqref{eq-3.12} at $x = 0$ and using the fact that $r_1^2\neq r_2^2$, we get $c_3=c_4=0$, therefore
\begin{equation*}
			y(x)=c_1 \sinh(r_1 x)+c_2\sinh(r_2 x).
			\end{equation*}
 Moreover, the boundary conditions in \eqref{eq-3.12} at $x=1$ can be expressed by  $$M \begin{pmatrix}
	c_1\\ c_2
\end{pmatrix}=0,$$ where
\begin{equation*}
	M=\begin{pmatrix}
	\sinh(r_1)& \sinh(r_2)\\ \noalign{\medskip}
	\left(r_1^2-\lambda^2\right)r_1 \cosh(r_1)+\gamma\la(\la+\eta)^{\alpha-1} r_1^2\sinh(r_1)& \left(r_2^2-\lambda^2\right)r_2\cosh(r_2)+\gamma\la(\la+\eta)^{\alpha-1} r_2^2\sinh(r_2)
	\end{pmatrix}.
\end{equation*}
Then the determinant of $M$ is given by
\begin{equation*}
det\left(M\right)=-i b\lambda^2\left(2 \gamma(\la+\eta)^{\alpha-1} \sinh(r_1)\sinh(r_2)+\frac{r_1}{\lambda}\cosh(r_1)\sinh(r_2) +\frac{r_2}{\lambda}\sinh(r_1) \cosh(r_2)\right).
	\end{equation*}
Equation \eqref{eq-3.12}  admits a non trivial solution if and only if $\displaystyle{det\left(M\right)}=0$; i.e., if and only if the eigenvalues of $\mathcal{A}$ are roots of the function $f$ defined by
\begin{equation}\label{eq-3.14}
f(\lambda)= 2\gamma(\la+\eta)^{\alpha-1} \sinh(r_1)\sinh(r_2)+\frac{r_1}{\lambda}\cosh(r_1)\sinh(r_2) +\frac{r_2}{\lambda}\sinh(r_1) \cosh(r_2).
\end{equation}
\begin{pro} \label{Theorem-3.2}
Assume that $a=1$ and that condition ${\rm (SC1)}$ holds. Then there exists $n_0\in \mathbb{N}$ sufficiently large  and two sequences $\left(\lambda_{1,n}\right)_{ |n|\geq n_0} $ and $\left(\lambda_{2,n}\right)_{ |n|\geq n_0} $ of simple roots of $f$ (that are also simple eigenvalues of $\mathcal{A}$) satisfying the following asymptotic behavior:\\[0.1in]
\textbf{Case 1.} If $b\neq k\pi$, $k\in\mathbb{Z}^*$, then
\begin{equation}\label{eq-3.15}
		\displaystyle{\lambda_{1,n}= i n\pi+{\frac {\gamma\, \left( 1-\cos \left( b \right)  \right)  \left( i\cos
 \left( \frac{\pi\alpha}{2} \right) -\sin \left( \frac{\pi\alpha}{2}
 \right)  \right) }{ 2\left( n\pi \right) ^{1-\alpha}}}+o \left( 
 \frac{1}{{n}^{1-\alpha}} \right) 
, \ \ \forall \ |n|\geq n_0}
\end{equation}
and
\begin{equation}\label{eq-3.16}
		\displaystyle{\lambda_{2,n}= i n\pi+\frac{i\pi}{2}}+{\frac {\gamma\, \left( 1+\cos \left( b \right)  \right)  \left( i\cos
 \left( \frac{\pi\alpha}{2} \right) -\sin \left( \frac{\pi\alpha}{2}
 \right)  \right) }{ 2\left( n\pi \right) ^{1-\alpha}}}+o \left( 
 \frac{1}{{n}^{1-\alpha}} \right) 
, \ \ \forall\  |n|\geq n_0.
\end{equation}
\textbf{Case 2.} If $b=2 k\pi$, $k\in\mathbb{Z}^*$, then
\begin{equation}\label{eq-3.17}
		\displaystyle{\lambda_{1,n}= i n\pi+{\frac {i{b}^{2}}{8n\pi}}+{\frac {7 i b^4}{128{\pi}^{
3}{n}^{3}}}+{\frac {\gamma\,{b}^{6} \left( i\cos \left(\frac{\pi\alpha}{2}  \right) -\sin \left( \frac{\pi\alpha}{2} \right)  \right) }{128\,{
\pi}^{5-\alpha}{n}^{5-\alpha}}}+O \left( \frac{1}{n^5}\right) 
, \ \ \forall \ |n|\geq n_0}
\end{equation}
and
\begin{equation}\label{eq-3.18}
		\displaystyle{\lambda_{2,n}= i n\pi+\frac{i\pi}{2}}+{\frac {\gamma \left( i\cos
 \left( \frac{\pi\alpha}{2} \right) -\sin \left( \frac{\pi\alpha}{2}
 \right)  \right) }{ \left( n\pi \right) ^{1-\alpha}}}+o \left( 
 \frac{1}{{n}^{1-\alpha}} \right) , \ \ \forall\  |n|\geq n_0.
\end{equation}
\textbf{Case 3.} If $b=(2 k+1)\pi$, $k\in\mathbb{Z}^*$, then
\begin{equation}\label{eq-3.19}
		\displaystyle{\lambda_{1,n}= i n\pi+{\frac {\gamma \left( i\cos
 \left( \frac{\pi\alpha}{2} \right) -\sin \left( \frac{\pi\alpha}{2}
 \right)  \right) }{ \left( n\pi \right) ^{1-\alpha}}}+o \left( 
 \frac{1}{{n}^{1-\alpha}} \right), \ \ \forall \ |n|\geq n_0}
\end{equation}
and
\begin{equation}\label{eq-3.20}
\begin{array}{ll}
		\displaystyle{\lambda_{2,n}= i n\pi+\frac{i\pi}{2}+{\frac {i{b}^{2}}{8n\pi}}-{\frac {i{b}^{2}}{16\pi {n}^{2}}}+{\frac 
{i{b}^{2} \left( 4{\pi}^{2}+7{b}^{2} \right) }{{128\pi
}^{3}{n}^{3}}}}
\\ \\ \hspace{1cm}
\displaystyle{-{\frac {i{b}^{2} \left( 4{\pi}^{2}+21
{b}^{2} \right) }{256{\pi}^{3}{n}^{4}}}+{\frac {\gamma {b}^{6} \left( 
i\cos \left(\frac{\pi\alpha}{2} \right) -\sin \left( \frac{\pi\alpha}{2}
 \right)  \right) }{256 {\pi}^{5-\alpha}{n}^{5-\alpha}}}+O
 \left(\frac{1}{n^5} \right) 
,\ \ \forall\ |n|\geq n_0.}
\end{array}
\end{equation}			
\end{pro}	
\begin{proof} The proof is divided into four steps.\\ [0.1in]
\noindent \textbf{Step 1.} In this step, we proof the following asymptotic behavior estimate 
			\begin{equation}\label{eq-3.21}
			\begin{array}{ll}
		\displaystyle{
		f(\lambda)=	{\frac {\gamma }{{\lambda}^{1-\alpha}}} \left( 1+O \left(\frac{1}{\lambda} \right)  \right) 
 \left( \cosh \left( r_1+r_2 \right) -\cosh \left( r_1-r_2\right)  \right)}
 \\ \\ \hspace{1cm}
\displaystyle{+ \left( 1+{
\frac {{b}^{2}}{8{\lambda}^{2}}}+O \left(\frac{1}{\lambda^3}\right) 
 \right) \sinh \left(r_1+r_2 \right)  - \left( {\frac {ib
}{2\lambda}}+O \left( \frac{1}{\lambda^3} \right)  \right) \sinh \left( r_1-r_2\right)}.
\end{array}
\end{equation}
From equation \eqref{eq-3.13} it follows that for $\lambda$ large enough, we have
\begin{equation}\label{eq-3.22}
\frac{r_1}{\lambda}=1+\frac{i b}{2\lambda}+\frac{b^2}{8\lambda^2}+O \left( \frac{1}{\lambda^3} \right) \ \ \ \text{and}\ \ \ \frac{r_2}{\lambda}=1-\frac{i b}{2\lambda}+\frac{b^2}{8\lambda^2}+O \left( \frac{1}{\lambda^3} \right). 
\end{equation}
Also, we have
\begin{equation}\label{eq-3.23}
\frac{1}{(\lambda+\eta)^{1-\alpha}}=\frac{1}{\lambda^{1-\alpha}}\left(1+O \left(\frac{1}{\lambda} \right)  \right). 
\end{equation}
Inserting \eqref{eq-3.22} and \eqref{eq-3.23} in \eqref{eq-3.14}, we get \eqref{eq-3.21}. \\[0.1in]
\noindent \textbf{Step 2.} First, from \eqref{eq-3.22} for $\la$ large enough, we have
\begin{equation}\label{eq-3.26}
r_1+r_2=2\lambda+\frac{b^2}{4\lambda}+O\left(\frac{1}{\lambda^2}\right)\ \ \ \text{and}\ \ \ r_1-r_2=ib +O\left(\frac{1}{\lambda^2}\right).
\end{equation}
From equation \eqref{eq-3.26} it follows for $\la\in S$ large enough, we have 
\begin{equation}\label{eq-3.27}
\left\{
\begin{array}{ll}
\displaystyle{\sinh(r_1+r_2)=\sinh(2\lambda)+\frac{b^2\, \cosh(2\lambda)}{4\lambda^2}+O\left(\frac{1}{\lambda^2}\right),}
\\ \noalign{\medskip}
\displaystyle{\sinh(r_1-r_2)=i\sin(b)+O\left(\frac{1}{\lambda^2}\right),}
\\ \noalign{\medskip}
\displaystyle{\cosh(r_1+r_2)-\cosh(r_1-r_2)=\cosh(2\lambda)-\cos(b)+O\left(\frac{1}{\lambda}\right). }
\end{array}
\right.
\end{equation}
So, inserting \eqref{eq-3.27} in \eqref{eq-3.21}, we get
\begin{equation}\label{eq-3.28}
f(\lambda)=\sinh(2\lambda)+\frac{\gamma\, \left(\cosh(2\lambda)-\cos(b)\right)}{\lambda^{1-\alpha}}+\frac{b\,\left(\cosh(2\lambda)\, b+2\sin(b)\right)}{4\lambda}+O\left(\frac{1}{\lambda^{2-\alpha}}\right).
\end{equation}
Since the roots of the analytic function $\la \to \sinh(2\lambda)$ are $\lambda^0_{1,n}=i n\pi$ and $\lambda^0_{2,n}=i n\pi+\frac{i\pi}{2}$, for any $n\in \Z$, using the Rouch\'e's Theorem, we deduce that $f$ admits an infinity of simple roots in $S$ denoted by $\lambda_{1,n}$ and $\lambda_{2,n}$ for $|n|\geq n_0$, for $n_0$ large enough, such that \begin{eqnarray}
\lambda_{1,n}=i n\pi +\epsilon_{1,n},\quad \text{where }\lim_{|n|\to+\infty}\epsilon_{1,n}=0,\label{eq-3.24}\\
\lambda_{2,n}=i n\pi+\frac{i\pi}{2} +\epsilon_{2,n},\quad \text{where }\lim_{|n|\to+\infty}\epsilon_{2,n}=0.\label{eq-3.25}
\end{eqnarray}\\[0.1in]
\noindent \textbf{Step 3.}  {\bf Asymptotic behavior of $\epsilon_{1,n}$}. First, using \eqref{eq-3.24} we obtain 
 \begin{equation}\label{fesinh}
 \sinh(2\la_{1,n})=\sinh(2\epsilon_{1,n})=2\epsilon_{1,n}+O\left(\epsilon_{1,n}^3\right),
 \end{equation}
\begin{equation}\label{fecosh}
 \cosh(2\la_{1,n})=\cosh(2\epsilon_{1,n})=1+O\left(\epsilon_{1,n}^2\right),
 \end{equation}
\begin{equation}\label{felam}
\dfrac{1}{\la_{1,n}}=-\dfrac{i}{n\pi}+O\left(\frac{\epsilon_{1,n}}{n^2}\right)
\end{equation}
and 
\begin{equation}\label{feplam}
\dfrac{1}{\la_{1,n}^{1-\alpha}}=\dfrac{1}{(in\pi)^{1-\alpha}}+O\left(\frac{\epsilon_{1,n}}{n^{2-\alpha}}\right).
\end{equation}
Next, by inserting \eqref{fesinh}-\eqref{feplam} in the identity $f(\la_{1,n})=0$ and using the fact that
$$i^{-1+\alpha}=\sin\left(\frac{\pi \alpha}{2}\right)-i\cos\left(\frac{\pi \alpha}{2}\right),$$
we find after simplification  
  \begin{equation}\label{eq-3.31}
\epsilon_{1,n}-\frac{\gamma\, \left(1-\cos(b)\right)\left(-\sin\left(\frac{\pi \alpha}{2}\right)+i\cos\left(\frac{\pi \alpha}{2}\right)\right)}{2(n \pi)^{1-\alpha}}+\frac{b\,\left( b+2\sin(b)\right)}{8in\pi}+O\left(\frac{1}{n^{2-\alpha}}\right)+O\left(\frac{\epsilon_{1,n}^2}{n^{1-\alpha}}\right)+O(\epsilon_{1,n}^3)=0
\end{equation}
and thus 
\begin{equation}\label{ep1n1}
\epsilon_{1,n}=\frac{\gamma\, \left(1-\cos(b)\right)\left(-\sin\left(\frac{\pi \alpha}{2}\right)+i\cos\left(\frac{\pi \alpha}{2}\right)\right)}{2(n \pi)^{1-\alpha}}+i\frac{b\,\left( b+2\sin(b)\right)}{8n\pi}+O\left(\frac{1}{n^{2-\alpha}}\right)+O\left(\frac{\epsilon_{1,n}^2}{n^{1-\alpha}}\right).
\end{equation}
 We distinguish two cases:\\[0.1in]
 \textbf{Case 1.} There exists  no integer $k\in \mathbb{Z}$ such that $b=2k\pi$. Then, we have
 $$1-\cos(b)\neq0,$$
 therefore, from \eqref{eq-3.31}, we get
 \begin{equation}\label{eq-3.32}
 \epsilon_{1,n}=\frac{\gamma\, \left(1-\cos(b)\right)\left(-\sin\left(\frac{\pi \alpha}{2}\right)+i\cos\left(\frac{\pi \alpha}{2}\right)\right)}{2(n \pi)^{1-\alpha}}+o\left(\frac{1}{n^{1-\alpha}}\right).
 \end{equation}
Inserting \eqref{eq-3.32} in \eqref{eq-3.24}, we get estimations \eqref{eq-3.15} and \eqref{eq-3.19}.\\[0.1in]
\textbf{Case 2.} If there exists $k\in\mathbb{Z}$ such that $b=2k\pi$, then
$$1-\cos(b)=0\ \ \ \text{and}\ \ \ \sin(b)=0,$$
 therefore, from \eqref{ep1n1}, we get
 \begin{equation}\label{eq-3.33}
\epsilon_{1,n}=\frac{ib^2}{8n\pi}+o\left(\frac{1}{n}\right).
\end{equation}
 Inserting \eqref{eq-3.33} in \eqref{eq-3.24}, we get
 \begin{equation}\label{eq-3.34}
\lambda_{1,n}=i n\pi +\frac{ib^2}{8n\pi}+\frac{\widetilde{\varepsilon}_{1,n}}{n},\quad \text{where }\lim_{|n|\to+\infty}\widetilde{\varepsilon}_{1,n}=0.
 \end{equation}
That implies that 
 \begin{equation*}
 r_1+r_2=2in\pi+\frac{2\widetilde{\varepsilon}_{1,n}}{n}-\frac{b^2\, (3i b^2-16\pi \widetilde{\varepsilon}_{1,n}) }{64 \pi n^3}+O \left( \frac{1}{n^5}\right)\ \ \ \text{and}\ \ \ r_1-r_2=i b+\frac{i b^3}{8\pi^2 n^2}+O \left( \frac{1}{n^4}\right)
 \end{equation*}
and thus 
\begin{equation}\label{eq-3.35}
\left\{
\begin{array}{ll}
\displaystyle{\sinh(r_1+r_2)=\frac{2\widetilde{\varepsilon}_{1,n}}{n}+\frac{256\pi^3\widetilde{\varepsilon}_{1,n}^3+48\pi b^2\widetilde{\varepsilon}_{1,n}-9ib^4}{192\pi^3n^3}+O\left(\frac{1}{n^5}\right),}
\\ \noalign{\medskip}
\displaystyle{\sinh(r_1-r_2)=\frac{i b^3}{8\pi^2 n^2}+O \left( \frac{1}{n^4}\right),}
\\ \noalign{\medskip}
\displaystyle{\cosh(r_1+r_2)-\cosh(r_1-r_2)=\frac{2\widetilde{\varepsilon}_{1,n}}{n^2}+O\left(\frac{1}{n^4}\right). }
\end{array}
\right.
\end{equation}
 Inserting \eqref{eq-3.34} and \eqref{eq-3.35} in \eqref{eq-3.21}, then using the asymptotic expansion, we get
 \begin{equation*}
 f(\lambda_{1,n})=\frac{2}{n}\left(\widetilde{\varepsilon}_{1,n}-\frac{7 i b^4}{128\pi^3 n^2} +O\left(\frac{1}{n^{3-\alpha}}\right)+O\left(\frac{\widetilde{\varepsilon}_{1,n}^2}{n^{2-\alpha}}\right) \right)=0,
 \end{equation*}
hence, we obtain
\begin{equation}\label{eq-3.36}
\widetilde{\varepsilon}_{1,n}=\frac{7 i b^4}{128\pi^3 n^2}+\frac{\zeta_{1,n}}{n^2},\ \text{such that }  \lim_{|n|\to+\infty}\zeta_{1,n}=0.
\end{equation}
Substituting  \eqref{eq-3.36} in \eqref{eq-3.34}, we get
 \begin{equation}\label{eq-3.37}
 \lambda_{1,n}=i n\pi +\frac{ib^2}{8n\pi}+\frac{7 i b^4}{128\pi^3 n^2}+\frac{\zeta_{1,n}}{n^3}, 
 \end{equation}
again  the  first real part of $\lambda_{1,n}$ still does not exists, so  we need to increase the order of the finite expansion.  For this aim, inserting \eqref{eq-3.37} in \eqref{eq-3.13} and using the asymptotic expansion, we get
 \begin{equation*}
 r_1+r_2=2in\pi+\frac{i b^4+32\pi^3 \zeta_{1,n} }{16 \pi^3 n^3}+O \left( \frac{1}{n^5}\right)\ \ \ \text{and}\ \ \ r_1-r_2=i b+\frac{i b^3}{8\pi^2 n^2}+O \left( \frac{1}{n^4}\right).
 \end{equation*}
Therefore, we have
\begin{equation}\label{eq-3.38}
\left\{
\begin{array}{ll}
\displaystyle{\sinh(r_1+r_2)=\frac{2\zeta_{1,n}+\frac{i b^4}{16\pi^3}}{n^3}+O\left(\frac{1}{n^5}\right),}
\\ \noalign{\medskip}
\displaystyle{\sinh(r_1-r_2)=\frac{i b^3}{8\pi^2 n^2}+O \left( \frac{1}{n^4}\right),}
\\ \noalign{\medskip}
\displaystyle{\cosh(r_1+r_2)-\cosh(r_1-r_2)=\frac{b^4}{128 \pi^4 n^4}+O\left(\frac{1}{n^6}\right). }
\end{array}
\right.
\end{equation}
 Inserting \eqref{eq-3.37} and \eqref{eq-3.38} in \eqref{eq-3.21}, then using the asymptotic expansion, we get
 \begin{equation*}
 \frac{2}{n^3}\left(\zeta_{1,n}-{\frac {\gamma\,{b}^{6} \left( i\cos \left(\frac{\pi\alpha}{2}  \right) -\sin \left( \frac{\pi\alpha}{2} \right)  \right) }{128\,{
\pi}^{5-\alpha}{n}^{2-\alpha}}}+O\left(\frac{1}{n^2}\right)\right)=0,
 \end{equation*}
hence, we obtain
\begin{equation}\label{eq-3.39}
\zeta_{1,n}={\frac {\gamma\,{b}^{6} \left( i\cos \left(\frac{\pi\alpha}{2}  \right) -\sin \left( \frac{\pi\alpha}{2} \right)  \right) }{128\,{
\pi}^{5-\alpha}{n}^{2-\alpha}}}+O\left(\frac{1}{n^2}\right).
\end{equation}
Inserting \eqref{eq-3.39} in \eqref{eq-3.37}, we get estimation \eqref{eq-3.17}. \\[0.1in]
\textbf{Step 4.}  {\bf Asymptotic behavior of $\epsilon_{2,n}$}. First, using \eqref{eq-3.25} we obtain 
 \begin{equation}\label{fesinh2}
 \sinh(2\la_{2,n})=-\sinh(2\epsilon_{2,n})=-2\epsilon_{2,n}+O\left(\epsilon_{2,n}^3\right),
 \end{equation}
\begin{equation}\label{fecosh2}
 \cosh(2\la_{2,n})=-\cosh(2\epsilon_{2,n})=-1+O\left(\epsilon_{2,n}^2\right),
 \end{equation}
\begin{equation}\label{felam2}
\dfrac{1}{\la_{2,n}}=-\dfrac{i}{n\pi}+O\left(\frac{1}{n^2}\right)
\end{equation}
and 
\begin{equation}\label{feplam2}
\dfrac{1}{\la_{2,n}^{1-\alpha}}=\dfrac{1}{(in\pi)^{1-\alpha}}+O\left(\frac{1}{n^{2-\alpha}}\right).
\end{equation}
Next, by inserting \eqref{fesinh2}-\eqref{feplam2} in the identity $f(\la_{1,n})=0$, we find after simplification  
\begin{equation}\label{eq-3.42}
\epsilon_{2,n}-\frac{\gamma\, \left(1+\cos(b)\right)\left(-\sin\left(\frac{\pi \alpha}{2}\right)+i\cos\left(\frac{\pi \alpha}{2}\right)\right)}{2(n \pi)^{1-\alpha}}+\frac{b\,\left( b-2\sin(b)\right)}{8in\pi}+O\left(\frac{1}{n^{2-\alpha}}\right)+O\left(\frac{\epsilon_{2,n}^2}{n^{1-\alpha}}\right)+O(\epsilon_{2,n}^3)=0
\end{equation}
and thus 
\begin{equation}\label{ep2n1}
\epsilon_{2,n}=\frac{\gamma\, \left(1+\cos(b)\right)\left(-\sin\left(\frac{\pi \alpha}{2}\right)+i\cos\left(\frac{\pi \alpha}{2}\right)\right)}{2(n \pi)^{1-\alpha}}+i\frac{b\,\left( b-2\sin(b)\right)}{8n\pi}+O\left(\frac{1}{n^{2-\alpha}}\right)+O\left(\frac{\epsilon_{1,n}^2}{n^{1-\alpha}}\right).
\end{equation}
We distinguish two cases:\\[0.1in]
 \textbf{Case 1.} There exists  no integer $k\in \mathbb{Z}$ such that $b=\pi+2k\pi$. Then, we have
 $$1+\cos(b)\neq0,$$
 therefore, from \eqref{eq-3.42}, we get
 \begin{equation}\label{eq-3.43}
 \epsilon_{2,n}=\frac{\gamma\, \left(1+\cos(b)\right)\left(-\sin\left(\frac{\pi \alpha}{2}\right)+i\cos\left(\frac{\pi \alpha}{2}\right)\right)}{2(n \pi)^{1-\alpha}}+o\left(\frac{1}{n^{1-\alpha}}\right).
 \end{equation}
Inserting \eqref{eq-3.43} in \eqref{eq-3.25}, we get estimations \eqref{eq-3.16} and \eqref{eq-3.18}.\\[0.1in]
\textbf{Case 2.} If there exists $k\in\mathbb{Z}$ such that $b=\pi+2k\pi$, then
$$1+\cos(b)=0\ \ \ \text{and}\ \ \ \sin(b)=0,$$
 therefore, from \eqref{eq-3.42}, we get
 \begin{equation}\label{eq-3.44}
\epsilon_{2,n}=\frac{ib^2}{8n\pi}+o\left(\frac{1}{n}\right).
\end{equation}
 Inserting \eqref{eq-3.44} in \eqref{eq-3.25}, we get
 \begin{equation}\label{eq-3.45}
\lambda_{2,n}=i n\pi +\frac{i\pi}{2}+\frac{ib^2}{8n\pi}+\frac{\widetilde{\varepsilon}_{2,n}}{n}\ \text{where} \lim_{|n|\to+\infty}\widetilde{\varepsilon}_{2,n}=0.
,
 \end{equation} 
Since in this case the  first real part of $\lambda_{2,n}$ still does not exists, we need to increase the order of the finite expansion.  Inserting \eqref{eq-3.45} in \eqref{eq-3.13} and using the asymptotic expansion, we get
 \begin{equation*}
 \left\{
 \begin{array}{ll}
 \displaystyle{r_1+r_2=2in\pi+i\pi+\frac{2\widetilde{\varepsilon}_{2,n}}{n}+\frac{i b^2}{8\pi n^2}+O \left( \frac{1}{n^3}\right)},\\ \noalign{\medskip}
\displaystyle{
 
  r_1-r_2=i b+\frac{i b^3}{8\pi^2 n^2}+O \left( \frac{1}{n^3}\right).}
 \end{array} 
 \right.
 \end{equation*}
 Therefore, we have
\begin{equation}\label{eq-3.46}
\left\{
\begin{array}{ll}
\displaystyle{\sinh(r_1+r_2)=-\frac{2\widetilde{\varepsilon}_{2,n}}{n}-\frac{i b^2}{8\pi n^2}+O \left( \frac{1}{n^3}\right),}
\\ \noalign{\medskip}
\displaystyle{\sinh(r_1-r_2)=-\frac{i b^3}{8\pi^2 n^2}+O \left( \frac{1}{n^3}\right),}
\\ \noalign{\medskip}
\displaystyle{\cosh(r_1+r_2)-\cosh(r_1-r_2)=-\frac{2\widetilde{\varepsilon}^2_{2,n}}{n^2}+O\left(\frac{1}{n^3}\right). }
\end{array}
\right.
\end{equation}
Inserting \eqref{eq-3.45} and \eqref{eq-3.46} in \eqref{eq-3.21}, then using the asymptotic expansion, we get
 \begin{equation*}
 -\frac{2}{n}\left(\widetilde{\varepsilon}_{2,n}+\frac{i b^2}{16\pi n}
+O\left(\frac{1}{n^2}\right)+O\left(\frac{\widetilde{\varepsilon}_{2,n}^2}{n^{2-\alpha}}\right) \right)=0,
 \end{equation*}
 hence, we get
 \begin{equation}\label{eq-3.47}
 \widetilde{\varepsilon}_{2,n}=-\frac{i b^2}{16\pi n}+\frac{\zeta_{2,n}}{n},\ \ \text{such that }\lim_{|n|\to+\infty}\zeta_{2,n}=0.
 \end{equation}
Inserting \eqref{eq-3.47} in \eqref{eq-3.45}, we get
  \begin{equation}\label{eq-3.48}
\lambda_{2,n}=i n\pi +\frac{i\pi}{2}+\frac{ib^2}{8n\pi}-\frac{i b^2}{16\pi n^2}+\frac{\zeta_{2,n}}{n^2},
 \end{equation} 
 again  the  first real part of $\lambda_{2,n}$ still does not exists, so  we need to increase the order of the finite expansion.  For this aim, inserting \eqref{eq-3.48} in \eqref{eq-3.13} and using the asymptotic expansion, we get
 \begin{equation*}
 \left\{
 \begin{array}{ll}
 \displaystyle{r_1+r_2=2in\pi+i\pi+\frac{2\zeta_{2,n}}{n}-\frac{i b^2\, (4\pi^2+3b^2) }{64 \pi n^3}+\frac{b^2\left(4i\pi^2+9i b^2+32\pi \zeta_{2,n} \right)}{128\pi^3 n^4}+O \left( \frac{1}{n^5}\right)},\\ \noalign{\medskip}
\displaystyle{
 
  r_1-r_2=i b+\frac{i b^3}{8\pi^2 n^2}-\frac{i b^3}{8\pi^2 n^3}+\frac{3i b^3\left(4\pi^2+b^2\right)}{128\pi^4 n^4}+O \left( \frac{1}{n^5}\right).}
 \end{array} 
 \right.
 \end{equation*}
Therefore, we have
\begin{equation}\label{eq-3.49}
\left\{
\begin{array}{ll}
\displaystyle{\sinh(r_1+r_2)=-\frac{2\zeta_{2,n}}{n}+\frac{i b^2\, (4\pi^2+3b^2) }{64 \pi n^3}-\frac{b^2\left(4i\pi^2+9i b^2+32\pi \zeta_{2,n} \right)}{128\pi^3 n^4}+O \left( \frac{1}{n^5}\right),}
\\ \noalign{\medskip}
\displaystyle{\sinh(r_1-r_2)=-\frac{i b^3}{8\pi^2 n^2}+\frac{i b^3}{8\pi^2 n^3}-\frac{3i b^3\left(4\pi^2+b^2\right)}{128\pi^4 n^4}+O \left( \frac{1}{n^5}\right),}
\\ \noalign{\medskip}
\displaystyle{\cosh(r_1+r_2)-\cosh(r_1-r_2)=-\frac{2\zeta^2_{2,n}+\frac{b^6}{128\pi^4}}{n^4}+O\left(\frac{1}{n^5}\right). }
\end{array}
\right.
\end{equation}
 Inserting \eqref{eq-3.48} and \eqref{eq-3.49} in \eqref{eq-3.21}, then using the asymptotic expansion, we get
 \begin{equation*}
 -\frac{2}{n^2}\left(\zeta_{2,n}-{\frac {i{b}^{2} \left( 4{\pi}^{2}+7{b}^{2}
 \right) }{128{\pi}^{3}n}}+{\frac {i{b}^{2} \left( 4{\pi
}^{2}+21{b}^{2} \right) }{256{\pi}^{3}{n}^{2}}}-{\frac {\gamma{b}^{6}
 \left( i\cos \left( \frac{\pi\alpha}{2} \right) -\sin \left( \frac{\pi\alpha}{2} \right)  \right) }{256{\pi}^{4}{\pi}^{1-\alpha}{n}^{3-\alpha}
}}
+O\left(\frac{1}{n^3}\right)+O\left(\frac{\zeta_{2,n}^2}{n^{3-\alpha}}\right) \right)=0,
 \end{equation*}
hence, we get
\begin{equation}\label{eq-3.50}
\zeta_{2,n}={\frac {i{b}^{2} \left( 4{\pi}^{2}+7{b}^{2}
 \right) }{128{\pi}^{3}n}}-{\frac {i{b}^{2} \left( 4{\pi
}^{2}+21{b}^{2} \right) }{256{\pi}^{3}{n}^{2}}}+{\frac {\gamma{b}^{6}
 \left( i\cos \left( \frac{\pi\alpha}{2} \right) -\sin \left( \frac{\pi\alpha}{2} \right)  \right) }{256{\pi}^{4}{\pi}^{1-\alpha}{n}^{3-\alpha}
}}
+O\left(\frac{1}{n^3}\right). 
\end{equation}
Finally, inserting \eqref{eq-3.50} in \eqref{eq-3.48}, we get \eqref{eq-3.20}. The proof is thus complete.
\end{proof}
 \noindent\textbf{Proof of  Theorem \ref{Theorem-3.1}.} From Proposition \ref{Theorem-3.2}, the operator $\mathcal{A}$ has two branches of eigenvalues, the energy corresponding to the first and second branch of eigenvalues  has no exponential decaying. Therefore the total energy of   system \eqref{eq-2.1}-\eqref{eq-2.8} has no exponential decaying when $a=1$. Thus, the proof is complete.	
	\subsection{Non Uniform Stability for $\mathbf{a\neq 1}$}\label{Section-3.2}
\noindent	In this part, we assume that $a\neq 1$, $\eta>0$ and $b$ verified condition ${\rm (SC)}$. Our goal is to show that \eqref{eq-2.1}-\eqref{eq-2.8} is not exponentially stable. This result is due to the fact that a subsequence of eigenvalues of $\mathcal{A}$ is close to the imaginary axis.
	\begin{theoreme}\label{Theorem-3.3}
	Assume that  $a\neq 1$ and that condition ${\rm (SC)}$ holds. Then,  the semigroup of contractions $e^{t\AA}$ is not uniformly (exponentially) stable in the energy space $\HH$.
		\end{theoreme}
 \noindent For the proof of Theorem \ref{Theorem-3.3}, we aim to show that an infinite number of eigenvalues of $\mathcal{A}$ approach the imaginary axis.  We determine the characteristic equation satisfied by the eigenvalues of $\mathcal{A}$. For the aim, let $\lambda\in \C^*$ be an eigenvalue of $\mathcal{A}$ and let $U=(u,\lambda u,y,\lambda y,\omega) \in D(\mathcal{A})$ be an associated eigenvector such that $\|U\|_{\mathcal{H}}=1$. Then, we have 
 \begin{eqnarray*}
    \la^2u-u_{xx}+b\la y=0,\\
    \la^2y-ay_{xx}-b\la u=0,\\
    \omega(\xi)=\frac{\mu(\xi)}{\la +\xi^2+\eta}\la u(1)
 \end{eqnarray*}
  with the following boundary conditions		
\begin{eqnarray*}
y(0)=y(1)=u(0)=0,
\\
u_x(1)=-\gamma\kappa(\alpha)\int_{\mathbb{R}}|\xi|^{\alpha-\frac{1}{2}}\omega(\xi)d\xi.
\end{eqnarray*} 
 Equivalently,   we get the following system
 \begin{equation}\label{eq-3.51}
\left\{
\begin{array}{ll}
\displaystyle{ay_{xxxx}-(a+1)\la^2y_{xx}+\la^2(\la^2+b^2)y=0,}
\\[0.1in]
\displaystyle{y(0)=y(1)=y_{xx}(0)=0,}
\\[0.1in]
\displaystyle{\ a y_{xxx}(1)+a\gamma\la(\la+\eta)^{\alpha-1}y_{xx}(1)-\la^2y_x(1)=0.}
\end{array}
\right.
\end{equation}
The general solution of equation \eqref{eq-3.51} is given by 
	\begin{equation*}
	y(x)=c_1 \sinh(r_1 x)+c_2\sinh(r_2 x)+c_3\cosh(r_1 x)+c_4\cosh(r_2 x),
	\end{equation*}
	where $c_1,\ c_2,\ c_3,\ c_4\in\mathbb{C}$, and
	\begin{equation}\label{eq-3.52}
	r_1=\la\sqrt{\dfrac{(a+1)+(a-1)\sqrt{1-\dfrac{4ab^2}{(a-1)^2\la^2}}}{2a}}\ \ \ \text{and}\ \ \
	r_2=\la\sqrt{\dfrac{(a+1)-(a-1)\sqrt{1-\dfrac{4ab^2}{(a-1)^2\la^2}}}{2a}}.	\end{equation}
Using the boundary conditions in \eqref{eq-3.51}  at $x = 0$ and the fact that $r_1^2\neq r_2^2$, we get $c_3=c_4=0$. Therefore
\begin{equation*}
			y(x)=c_1 \sinh(r_1 x)+c_2\sinh(r_2 x).
			\end{equation*}
Moreover, the boundary conditions in \eqref{eq-3.51} at $x=1$ can be expressed by  $$\mathcal{M} \begin{pmatrix}
	c_1\\ c_2
\end{pmatrix}=0,$$ where
\begin{equation*}
	\mathcal{M}=\begin{pmatrix}
	\sinh(r_1)& \sinh(r_2)\\ \noalign{\medskip}
	\left(a r_1^2-\lambda^2\right)r_1 \cosh(r_1)+a\gamma\la(\la+\eta)^{\alpha-1} r_1^2\sinh(r_1)& \left(a r_2^2-\lambda^2\right)r_2\cosh(r_2)+a\gamma\la(\la+\eta)^{\alpha-1} r_2^2\sinh(r_2)
	\end{pmatrix}.
\end{equation*}
Then the determinant of $\mathcal{M}$ is given by
\begin{equation*}
det\left(\mathcal{M}\right)=-a \gamma\lambda  (\la+\eta)^{\alpha-1}(r_1^2-r_2^2) \sinh(r_1)\sinh(r_2)-r_1(ar_1^2-\lambda^2)\cosh(r_1)\sinh(r_2) +r_2(ar_2^2-\lambda^2)\sinh(r_1) \cosh(r_2).
	\end{equation*}
Hence a non trivial solution $y$ of \eqref{eq-3.51} exists if and only if $\displaystyle{det\left(\mathcal{M}\right)}=0$. Set $F(\la)=\displaystyle{\dfrac{det\left(\mathcal{M}\right)}{{(a-1)}{\lambda^3}}}$, thus the characteristic equation is equivalent to $F(\la)=0$.\\[0.1in] 
In the sequel, since $\mathcal{A}$ is dissipative, we study the asymptotic behavior of the large eigenvalues $\la$ of $\mathcal{A}$ in the strip $-\alpha_0\leq \Re(\la)\leq 0$, for some $\alpha_0>0$ large enough.
   \begin{pro}\label{Theorem-3.4}
	Assume that $a\neq1$ and that condition ${\rm (SC)}$ holds. Then there exists $n_0\in \mathbb{N}$ sufficiently large  and two sequences $\left(\lambda_{1,n}\right)_{ |n|\geq n_0} $ and $\left(\lambda_{2,n}\right)_{ |n|\geq n_0} $ of simple roots of $f$ (that are also simple eigenvalues of $\mathcal{A}$) satisfying the following asymptotic behavior
	\begin{equation}\label{eq-3.53}
	\la_{1,n}=in\pi\sqrt{a}+o(1) \quad \text{and}\quad \la_{2,n}=i\left(n+\frac{1}{2}\right)\pi+o(1).
	\end{equation}
\end{pro}
\begin{proof} The proof is divided into two steps.\\ [0.1in]
		\textbf{Step 1.}
			In this step, we proof the following asymptotic behavior estimate 
\begin{equation}\label{eq-3.54}
F(\lambda)=F_0(\lambda)+\frac{F_1(\lambda)}{\lambda^{1-\alpha}}+O\left(\frac{1}{\lambda}\right),
\end{equation}
where $$F_0(\lambda)=\cosh(\lambda)\sinh\left(\frac{\lambda}{\sqrt{a}}\right)\ \ \ \text{and}\ \ \		F_1(\lambda)=\sinh(\lambda)\sinh\left(\frac{\lambda}{\sqrt{a}}\right).$$
From equation \eqref{eq-3.52} it follows that for $\lambda$ large enough, we have
\begin{equation}\label{eq-3.55}
r_1=\lambda+O\left(\frac{1}{\lambda}\right)\ \ \ \text{and}\ \ \ r_2=\frac{\lambda}{\sqrt{a}}+O\left(\frac{1}{\lambda}\right).
\end{equation}
From \eqref{eq-3.55} and using the fact that real part of $\lambda$ is bounded, we get
\begin{equation}\label{eq-3.56}
\left\{
\begin{array}{ll}
\displaystyle{\sinh(r_1) \cosh(r_2) =\sinh(\lambda)\cosh\left(\frac{\lambda}{\sqrt{a}}\right)+O\left(\frac{1}{\lambda}\right),}
\\ \noalign{\medskip}
\displaystyle{  \cosh(r_1)\sinh(r_2) =\cosh(\lambda)\sinh\left(\frac{\lambda}{\sqrt{a}}\right)+O\left(\frac{1}{\lambda}\right),}
\\ \noalign{\medskip}
\displaystyle{\sinh(r_1)\sinh(r_2)=\sinh(\lambda)\sinh\left(\frac{\lambda}{\sqrt{a}}\right)+O\left(\frac{1}{\lambda}\right). }
\end{array}
\right.
\end{equation}	
Inserting 	\eqref{eq-3.55} and \eqref{eq-3.56} in $\det(\mathcal{M})$, then using the asymptotic expansion, we get
\begin{equation*}
det\left(\mathcal{M}\right)=-{(a-1)}{\lambda^3}\left(\cosh(\lambda)\sinh\left(\frac{\lambda}{\sqrt{a}}\right)+\frac{\sinh(\lambda)\sinh\left(\frac{\lambda}{\sqrt{a}}\right)}{\lambda^{1-\alpha}}+O\left(\frac{1}{\lambda}\right)\right).
\end{equation*} 
\\[0.1in]
\textbf{Step 2.} We look at the roots of $F(\la)$. First, the roots of $F_0$ are given by 
	$$
in\pi\sqrt{a}\ \ \ \text{and/ or}\ \ \ 	i\left(n+\frac{1}{2}\right)\pi\qquad n\in \Z.
	$$
Now, with help of Rouch\'e's Theorem, and using the asymptotic equation \eqref{eq-3.54}, it easy to see that the large roots of $F(\la)$ (denoted by $\lambda_{1,n}$ and $\lambda_{2,n}$) are simple and close to those of $F_0(\lambda)$; i.e., there exists $n_0\in \mathbb{N}$, such that for all $|n|>n_0$, we have
 $$   \la_{1,n}=in\pi\sqrt{a}+o(1)$$ \text{and/or} $$\la_{2,n}=i\left(n+\frac{1}{2}\right)\pi+o(1),
    $$
hence we get \eqref{eq-3.53}.
The proof is thus complete.
\end{proof}
 \noindent\textbf{Proof of  Theorem \ref{Theorem-3.3}.} From Proposition \ref{Theorem-3.4}, the operator $\mathcal{A}$ has two branches of eigenvalues, the energy corresponding to the first and second branch of eigenvalues  has no exponential decaying. Therefore the total energy of   system \eqref{eq-2.1}-\eqref{eq-2.8} has no exponential decaying when $a\neq1$. Thus, the proof is complete.
 \newpage
\section{Polynomial Stability}\label{Section-4}
\subsection{Polynomial Stability in the case $\mathbf{\eta>0}$ and  $\mathbf{a=1}$.}\label{Section-4.1}
\noindent From the subsection \ref{Section-3.1}, if $\eta>0$, $a=1$ and  $b$ satisfies condition ${\rm (SC1)}$, then system \eqref{eq-2.1}-\eqref{eq-2.8} is not uniformly (exponentially) stable, so it is natural to hope for a polynomial stability. For that purpose, we will use a frequency domain approach, namely we will use Theorem  2.4 of \cite{borichev:10} (see also \cite{batkai:06,batty:08,liurao:05}) that we partially recall.
\begin{lemma}\label{bt}
Let $(T(t))_{t\geq 0}$ be a bounded $C_0$-semigroup on a Hilbert space
$H$ with generator $A$ such that $i\R \subset\rho(A)$. Then for a fixed $\ell> 0$ the following
conditions are equivalent
\begin{eqnarray}
\label{(i)}
\|(is-A)^{-1}\| = O(|s|^\ell),\, s \to \infty,\\
\label{(ii)}
\|T(t)A^{-1}\| = O(t^{-1/\ell}),\, t\to \infty.
\end{eqnarray}
\end{lemma}
The aim of this section is to prove the following result:

\begin{theoreme}\label{Theorem-4.1}
Assume that $\eta>0$, $a=1$ and that condition ${\rm (SC1)}$ holds. Then, for all initial data $U_0=(u_0,u_1,y_0,y_1,\omega_0)\in D(\AA)$, there exists a constant $c>0$ independent of $U_0$, such that the energy of system \eqref{eq-2.1}-\eqref{eq-2.8} satisfies the following estimation:
\begin{equation}\label{eq-4.1}
E(t)\leq c\frac{1}{t^{p(\alpha)}}\|U_0\|^2_{D(\AA)},\quad \forall t>0,
\end{equation}
where 
$$
p(\alpha)=\left\{\begin{array}{lll}
\displaystyle
\frac{2}{1-\alpha}&\text{if}&b\in \pi \Z,
\\ \noalign{\medskip}
\displaystyle
\frac{2}{5-\alpha}&\text{if}&b\notin \pi \Z.
\end{array}\right.
$$ 
In addition, the energy decay rate \eqref{eq-4.1} is optimal in the sense that for any $\varepsilon >0$, we can not expect the decay rate $\frac{1}{t^{p(\alpha)+\varepsilon}}$ for all initial data $U_0 \in D(\AA)$.
\end{theoreme}
\noindent Since condition $i\R \subset\rho(A)$ was already checked in Theorem \ref{Theorem-2.4} in the case $\eta>0$, then the proof of  Theorem \ref{Theorem-4.1} is reduced to show that condition \eqref{(i)} holds 
with
$$
\ell=\left\{\begin{array}{lll}
1-\alpha&\text{if}&b\notin \pi\Z,
\\ \noalign{\medskip}
5-\alpha&\text{if}&b\in \pi\Z.
\end{array}
\right.
$$
This is checked by using a contradiction  argument. Indeed assume that it does not hold, then  there exist a sequence $\la_{n}\in\mathbb{R}, n\in \mathbb{N}$ 
such that $\la_{n}\to +\infty $ as $n\to +\infty$, and a sequence $U_{n}=(u^n,v^n,y^n,z^n,\omega^n) \in D(\AA)$  such that 
\begin{equation}\label{eq-4.2}
\|U^n\|_{\HH}=\|\left(u^n,v^n,y^n,z^n,\omega^n\right)\|_{\HH}=1,
\end{equation}
\begin{equation}\label{eq-4.3}
\la_n^{\ell}\left(i\la_nI-\AA\right)U^n=\left(f_1^n,g_1^n,f_2^n,g_2^n,f_3^n\right)\rightarrow 0\ \ \text{in}\ \ \HH.
\end{equation}
For simplificity, we replace $\la_{n}$ by $\la$;
  $U_{n}=(u^{n},v^{n},y^n,z^{n},\omega^n)$ by $U=(u,v,y,z,\omega)$ and   $F_{n}=\la_{n}^{l}(i\la_{n}I-\mathcal{A})U_{n}=(f_{1}^n,g_{1}^n,f_{2}^n,g_{2}^n,f_3^n)$ by $F=(f_{1},g_{1},f_{2},g_{2},f_3)$. Next, by detailing \eqref{eq-4.3}, we obtain 
\begin{eqnarray}
i\la u-v&=&\frac{f_1}{\la^{\ell}}\longrightarrow 0\ \ \text{in}\ \ H_L^1(0,1),\label{eq-4.4}\\
i\la v-u_{xx}+bz&=&\frac{g_1}{\la^{\ell}}\longrightarrow 0\ \ \text{in}\ \ L^2(0,1),\label{eq-4.5}\\
i\la y-z&=&\frac{f_2}{\la^{\ell}}\longrightarrow 0\ \ \text{in}\ \ H_0^1(0,1),\label{eq-4.6}\\
i\la z-y_{xx}-bv&=&\frac{g_2}{\la_n^{\ell}}\longrightarrow 0\ \ \text{in}\ \ L^2(0,1),\label{eq-4.7}\\
(i\la+\xi^2+\eta)\omega(\xi)-v(1)\mu(\xi)&=&\frac{f_3(\xi)}{\la^{\ell}}\longrightarrow 0\ \ \text{in}\ \ L^2(\R)\label{eq-4.8}.
\end{eqnarray}
First, multiplying equation \eqref{eq-4.3} by $U$ in $\mathcal{H}$, we get
\begin{equation}\label{eq-4.9}
-\gamma\kappa(\alpha)\int_{\R}(\xi^2+\eta)|\omega(\xi)|^2d\xi=\Re\left(\left<\left(i\la I-\AA\right)U,U\right>_{\HH}\right)=\frac{o(1)}{\la^{\ell}}.
\end{equation}
 For clarity, we divide the proof into several Lemmas. 
\begin{lemma}\label{Theorem-4.2} 
\noindent Assume that $\eta>0$, $a=1$ and that condition ${\rm (SC1)}$ holds. Then, the solution $(u,v,y,z,\omega)\in D(\AA)$ of \eqref{eq-4.4}-\eqref{eq-4.8} satisfies the following asymptotic behavior estimation
\begin{eqnarray}
\|u\|&=&\frac{O(1)}{\la},\label{eq-4.10}\\
\|y\|&=&\frac{O(1)}{\la},\label{eq-4.11}\\
|u_x(1)|&=&\frac{o(1)}{\la^{\frac{\ell}{2}}},\label{eq-4.12}\\
|u(1)|&=&\frac{o(1)}{\la^{\frac{\ell+\alpha+1}{2}}}.\label{eq-4.13}
\end{eqnarray}  
\end{lemma}
\begin{proof}
First, since $v$ and $z$ are uniformly bounded in $L^2(0,1)$, then from equations \eqref{eq-4.4} and \eqref{eq-4.6}, we deduce directly the estimations \eqref{eq-4.10} and \eqref{eq-4.11}. Next, from the boundary condition 
$$
u_x(1)+\gamma\kappa(\alpha)\int_{\R}\mu(\xi)\omega(\xi)d\xi=0,
$$
we get 
\begin{equation}\label{eq-4.14}
\left|u_x(1)\right|\leq \gamma\kappa(\alpha)\left(\int_{\R}\frac{\mu^2(\xi)}{\xi^2+\eta}d\xi\right)^{\frac{1}{2}}\left(\int_{\R}\left(\xi^2+\eta\right)|\omega(\xi)|^2d\xi\right)^{\frac{1}{2}}.
\end{equation}
Then, combining equation \eqref{eq-4.9} and \eqref{eq-4.14}, we obtain the desired estimation \eqref{eq-4.12}. Next, from \eqref{eq-4.8}, we get
\begin{equation}\label{eq-4.15}
\left|v(1)\right| \left|\xi\right|^{\alpha-\frac{1}{2}} \leq\left(\left|\lambda\right|+\xi^2+\eta\right)\left|\omega(\xi)\right|+\frac{1}{\left|\lambda\right|^{\ell}}\left|f_{3}(\xi)\right|.
\end{equation}
Multiplying equation \eqref{eq-4.15} by $(\left|\lambda\right|+\xi^2+\eta)^{-2}\left|\xi\right|$, integrating over $\mathbb{R}$ with respect to the variable $\xi$ and applying Cauchy-Shariwz inequality, we obtain
\begin{equation}\label{eq-4.16}
\left|v(1)\right| A_1  \leq A_2\left(\int_{\mathbb{R}}\left|\xi\omega(\xi)\right|^2d\xi\right)^{\frac{1}{2}}+\frac{A_3}{\left|\lambda\right|^{\ell}}\left(\int_{\mathbb{R}}\left|f_{3}(\xi)\right|^2d\xi\right)^{\frac{1}{2}},
\end{equation}
where
$$A_1=\int_{\mathbb{R}}\frac{ \left|\xi\right|^{\alpha+\frac{1}{2}}}{(\left|\lambda\right|+\xi^2+\eta)^{2}}d\xi,\ A_2=\left(\int_{\mathbb{R}}\frac{1}{\left(\left|\lambda\right|+\xi^2+\eta\right)^2}d\xi\right)^{\frac{1}{2}},\ A_3=\left(\int_{\mathbb{R}}\frac{ \xi^2}{(\left|\lambda\right|+\xi^2+\eta)^{4}} d\xi\right)^{\frac{1}{2}}.$$
It easy to check that 
\begin{equation}\label{eq-4.17}
A_2=\sqrt{\frac{\pi}{2}}\frac{1}{\left(\left|\lambda\right|+\eta\right)^{\frac{3}{4}}}\ \ \ \text{and}\ \ \ A_3=\frac{\sqrt{\pi}}{4}\frac{1}{\left(\left|\lambda\right|+\eta\right)^{\frac{5}{4}}}.
\end{equation}
Moreover, we have
\begin{equation}\label{eq-4.18}
I_1=\frac{2}{\left(\left|\lambda\right|+\eta\right)^{2}}\int_{0}^{\infty}\frac{ \xi^{\alpha+\frac{1}{2}}}{\left(1+\frac{\xi^2}{\left|\lambda\right|+\eta}\right)^{2}}d\xi.
\end{equation}
Thus equation \eqref{eq-4.18} may be simplified by defining a new variable $y=1+\frac{\xi^2}{\left|\lambda\right|+\eta}$.  Substituting $\xi$  by $\left(y-1\right)^{\frac{1}{2}}\left(\left|\lambda\right|+\eta\right)^{\frac{1}{2}}$ in equation \eqref{eq-4.18}, we get
\begin{equation*}
I_1=\left(\left|\lambda\right|+\eta\right)^{\frac{\alpha}{2}-\frac{5}{4}}\int_{1}^{\infty}\frac{\left(y-1\right)^{\frac{\alpha}{2}-\frac{1}{4}}}{y^2}dy.
\end{equation*}
Using the fact that $\alpha\in ]0,1[$, it easy to see that $y^{-2}\left(y-1\right)^{\frac{\alpha}{2}-\frac{1}{4}}\in L^1(1,+\infty)$, therefore we have
\begin{equation}\label{eq-4.19}
A_1=c_1\left(|\lambda|+\eta\right)^{\frac{\alpha}{2}-\frac{5}{4}}.
\end{equation}
where $c_1$ is a positive constant number. Inserting \eqref{eq-4.17} and  \eqref{eq-4.19} in  \eqref{eq-4.16}, then using \eqref{eq-4.3} and \eqref{eq-4.9},   we deduce that
\begin{equation}\label{eq-4.20}
 \left|v(1)\right|   \leq \sqrt{\frac{\pi}{2}}\frac{1}{c_1\left(|\lambda|+\eta\right)^{\frac{\alpha}{2}-\frac{1}{2}} } \frac{o\left(1\right)}{|\lambda|^{\frac{\ell}{2}}}+\frac{\sqrt{\pi}}{4}\frac{1}{c_1\left(|\lambda|+\eta\right)^{\frac{\alpha}{2}}}\frac{o(1)}{\left|\lambda\right|^{\ell}}.
\end{equation}
Since $\alpha\in ]0,1[$ and $\ell>0$, we have $\min\left(\frac{\ell+\alpha-1}{2},\ell+\frac{\alpha}{2},\ell\right)=\frac{\ell+\alpha-1}{2}$, hence from \eqref{eq-4.20}, we get
\begin{equation}\label{eq-4.21}
|v(1)|=\frac{o(1)}{\lambda^{\frac{\ell+\alpha-1}{2}}}.
\end{equation} 
Finally, combining equation equations \eqref{eq-4.4} and \eqref{eq-4.21}, we get the desired estimation \eqref{eq-4.13}. The proof is thus complete.
\end{proof}
\begin{lemma}\label{Theorem-4.3}
Let $h\in W^{1,\infty}(0,1)$. Assume that $\eta>0$, $a=1$, and condition ${\rm (SC1)}$ holds. Then, the solution $U=(u,v,y,z,\omega)\in D(\AA)$ of system \eqref{eq-4.4}-\eqref{eq-4.8} satisfies the following estimation 
\begin{equation}\label{eq-4.22}
-\int_0^1h'\left(|\la u|^2+|\la y|^2+|u_x|^2+|y_x|^2\right)dx+h(1)|y_x(1)|^2-h(0)|y_x(0)|^2-h(0)|u_x(0)|^2=\frac{o(1)}{\la^{\ell-1+\alpha}}+\frac{O(1)}{\la}.
\end{equation}
In particular, we have 
\begin{equation}\label{eq-4.27}
\left|y_x(1)\right|^2-\left|y_x(0)\right|^2-\left|u_x(0)\right|^2=\frac{o(1)}{\la^{\ell-1+\alpha}}.
\end{equation}
\end{lemma}
\begin{proof}
Substitute $v$ and $z$ in equations \eqref{eq-4.5} and \eqref{eq-4.7} by \eqref{eq-4.4} and \eqref{eq-4.6} respectively, we obtain the following system 
\begin{eqnarray}
\la^2 u+u_{xx}-i\la b y&=&-\frac{g_1+i\la f_1+bf_2}{\la^{\ell}},\label{eq-4.23}\\
\la^2 y+y_{xx}+i\la bu&=&-\frac{g_2+i\la f_2-bf_1}{\la^{\ell}}. \label{eq-4.24}
\end{eqnarray}
Multiplying equation \eqref{eq-4.23} by $2h\overline{u_x}$, integrating by parts and using Lemma \ref{Theorem-4.2}, we get 
\begin{equation}\label{eq-4.25}
-\int_0^1h'|\la u|^2dx-\int_0^1h'|u_x|^2dx-h(0)|u_x(0)|^2+2\Re\left\{i\la b\int_0^1hy_x\bar{u}dx\right\}=\frac{O(1)}{\la}+\frac{o(1)}{\la^{\ell-1+\alpha}}.
\end{equation}
Similarly, multiplying equation \eqref{eq-4.24} by $2h\overline{y_x}$, integrating by parts and using Lemma \ref{Theorem-4.2}, we get 
\begin{equation}\label{eq-4.26}
-\int_0^1h'|\la y|^2-\int_0^1h'|y_x|^2+h(1)|y_x(1)|^2-h(0)|y_x(0)|^2dx+2\Re\left\{i\la b\int_0^1hu\overline{y_x}dx\right\}=\frac{o(1)}{\la^{\ell}}.
\end{equation}
Note that, since $f_2$ converges to zero in $H_0^1(0,1)$ and $\la y$ is uniformly bounded in $L^2(0,1)$, then 
$$
\frac{1}{\la^{\ell}}\int_0^1\la f_2h\overline{y_x}dx=-\frac{1}{\la^{\ell}}\int_0^1\la\overline{y}\left((f_2)_xh+f_2h_x\right)dx=\frac{o(1)}{\la^{\ell}}.
$$
Combining equations \eqref{eq-4.25} and \eqref{eq-4.26}, we get the desired estimation \eqref{eq-4.22}. Finally, by tacking $h=1$ in equation \eqref{eq-4.22}, we obtain \eqref{eq-4.27}. The proof is thus complete.
\end{proof}
 
\begin{lemma}\label{Theorem-4.4}
Assume that $\eta>0$, $a=1$ and that condition ${\rm (SC1)}$ holds. Then, the solution $U=(u,v,y,z,\omega)\in D(\AA)$ of system \eqref{eq-4.23}-\eqref{eq-4.24} satisfies the following estimation 
\begin{equation*}
\left|y_x(1)\right|^2=\frac{o(1)}{\la^{\delta(\alpha)}},
\end{equation*}
where
$$
\delta(\alpha)=\left\{\begin{array}{lll}
\ell-1+\alpha&\text{if}&b\notin \pi\Z,\\
\ell-5+\alpha&\text{if}&b\in \pi\Z.
\end{array}
\right.
$$
\end{lemma}
\begin{proof}
Let $Y=(u,u_x,y,y_x) $, then system \eqref{eq-4.23}-\eqref{eq-4.24} could be written as 
\begin{equation}\label{eq-4.28}
Y_x=BY+G+\la F,
\end{equation}
where
\begin{equation}\label{eq-4.29}
B=\begin{pmatrix}
	0&1&0&0\\
	-\la^2&0&i\la b&0\\
	0&0&0&1\\
	-i\la b&0&-\la^2&0
	\end{pmatrix} ,\ \ G=\left(G_j\right)=\begin{pmatrix}
	0\\ \displaystyle
	\frac{-g_1-ib f_2}{\la^{\ell}}\\ 0\\ \displaystyle
	\frac{-g_2+ib f_1}{\la^{\ell}} 
	\end{pmatrix}\quad \text{and}\ \ F=\left(F_J\right)=\begin{pmatrix}
	0\\ \displaystyle
	-\frac{if_1}{\la^{\ell}}\\ 0\\ \displaystyle
	-\frac{if_2}{\la^{\ell}}
	\end{pmatrix}.
\end{equation}
Using Ordinary Differential Equation Theory, the solution of equation \eqref{eq-4.28} is given by 
\begin{equation*}
Y(x)=e^{Bx}Y_0+\int_0^xe^{B(x-z)}G(z)dz+\la\int_0^xe^{B(x-z)}F(z)dz.
\end{equation*}
Then, we have
$$
Y(1)=e^{B}Y_0+\int_0^1e^{B(1-z)}G(z)dz+\la\int_0^xe^{B(1-z)}F(z)dz.
$$
Equivalently, we get 
\begin{equation*}
e^{-B}Y(1)=Y_0+e^{-B}\int_0^1e^{B(1-z)}G(z)dz+\la e^{-B}\int_0^1e^{B(1-z)}F(z)dz.
\end{equation*}
Using Lemma \ref{Theorem-4.2}, we have 
$$
Y(1)=\left(\frac{o(1)}{\la^{\frac{\ell+\alpha+1}{2}}},\frac{o(1)}{\la^{\frac{\ell}{2}}},0,y_x(1)\right).
$$
Performing advanced calculation for the exponential of matrix $B$ and $-B$, we obtain the following matrix 
		$$e^B=\begin{pmatrix}
				A_1&\frac{1}{\la}A_2-\frac{b}{2\la^2}\left(A_4+\frac{b}{4}A_1\right)&iA_3&\frac{-i}{\la}A_4-\frac{ib}{2\la^2}\left(\frac{b}{4}A_3-A_2\right)\\
				\frac{b}{2}\left(\frac{b}{4}A_1-A_4\right)&A_1&\frac{ib}{2}\left(A_2+\frac{b}{4}A_3\right)&iA_3\\
				-iA_3&\frac{i}{\la}A_4+\frac{ib}{2\la^2}\left(\frac{b}{4}A_3-A_2\right)&A_1&\frac{1}{\la}A_2-\frac{b}{2\la^2}\left(A_4+\frac{b}{4}A_1\right)\\
				-\frac{ib}{2}\left(A_2+\frac{b}{4}A_3\right)&-iA_3&\frac{b}{2}\left(\frac{b}{4}A_1-A_4\right)&A_1
				\end{pmatrix}$$
			$$
			+\begin{pmatrix}
			0&0&0&0\\
			-\la A_2&0&i\la A_4&0\\
			0&0&0&0\\
			i\la A_4&0&-\la A_2&0
			\end{pmatrix}+\begin{pmatrix}
			o(1)&O\left(\frac{1}{\la^3}\right)&o(1)&O\left(\frac{1}{\la^3}\right)\\
			o(1)&o(1)&o(1)&o(1)\\
			o(1)&O\left(\frac{1}{\la^3}\right)&o(1)&O\left(\frac{1}{\la^3}\right)\\
			o(1)&o(1)&o(1)&o(1)
			\end{pmatrix},
			$$
			and the expression of $e^{-B}$ is given by 
			$$e^{-B}=\begin{pmatrix}
					A_1&-\frac{1}{\la}A_2+\frac{b}{2\la^2}\left(A_4+\frac{b}{4}A_1\right)&iA_3&\frac{i}{\la}A_4+\frac{ib}{2\la^2}\left(\frac{b}{4}A_3+A_2\right)\\
					-\frac{b}{2}\left(\frac{b}{4}A_1-A_4\right)&A_1&-\frac{ib}{2}\left(A_2+\frac{b}{4}A_3\right)&iA_3\\
					iA_3&\frac{i}{\la}A_4+\frac{ib}{2\la^2}\left(\frac{b}{4}A_3-A_2\right)&A_1&-\frac{1}{\la}A_2+\frac{b}{2\la^2}\left(A_4+\frac{b}{4}A_1\right)\\
					-\frac{ib}{2}\left(A_2+\frac{b}{4}A_3\right)&iA_3&-\frac{b}{2}\left(\frac{b}{4}A_1-A_4\right)&A_1
					\end{pmatrix}$$
				$$
				+\begin{pmatrix}
				0&0&0&0\\
				\la A_2&0&-i\la A_4&0\\
				0&0&0&0\\
				i\la A_4&0&+\la A_2&0
				\end{pmatrix}+\begin{pmatrix}
				o(1)&O\left(\frac{1}{\la^3}\right)&o(1)&O\left(\frac{1}{\la^3}\right)\\
				o(1)&o(1)&o(1)&o(1)\\
				o(1)&O\left(\frac{1}{\la^3}\right)&o(1)&O\left(\frac{1}{\la^3}\right)\\
				o(1)&o(1)&o(1)&o(1)
				\end{pmatrix},
				$$
where 
$$
A_1=\cos(\la)\cos\left(\frac{b}{2}\right),\ \ A_2=\sin(\la)\cos\left(\frac{b}{2}\right),\ \ A_3=\sin(\la)\sin\left(\frac{b}{2}\right),\ \ A_4=\cos(\la)\sin\left(\frac{b}{2}\right).
$$
Since $G_1=G_3=F_1=F_3=0$ and $A_j$, $j=1,2,3,4,$, are uniformly bounded then, from $e^{-B}$ and \eqref{eq-4.29}, we get 
\begin{equation}\label{eq-4.30}
e^{-B}\int_0^1e^{B(1-z)}\left(G(z)+\la F(z)\right)dz=\left(\frac{o(1)}{\la^{\ell}},\frac{o(1)}{\la^{\ell}},\frac{o(1)}{\la^{\ell}},\frac{o(1)}{\la^{\ell}}\right).
\end{equation}
So, from equations \eqref{eq-4.4} and \eqref{eq-4.30}, we get
\begin{equation}\label{eq-4.31}
e^{-B}Y(1)=Y(0)+\left(\frac{o(1)}{\la^{\ell}},\frac{o(1)}{\la^{\ell}},\frac{o(1)}{\la^{\ell}},\frac{o(1)}{\la^{\ell}}\right).
\end{equation}
We need distinguish two cases:\\[0.1in]
\textbf{Case 1}. If $b\neq k\pi$, $k\in \Z^\star$, then 
\begin{equation}\label{eq-4.32}
\sin\left(\frac{b}{2}\right)\neq0\ \ \ \text{and}\ \ \ \cos\left(\frac{b}{2}\right)\neq0.
\end{equation}
First, using the expression of $e^{-B}$ and equation \eqref{eq-4.31}, we get 
\begin{equation}\label{eq-4.33}
\frac{1}{\la}A_4y_x(1)=\frac{o(1)}{\la^{\frac{\ell+\alpha+1}{2}}}\ \ \ \text{and}\ \ \
\frac{1}{\la}A_2y_x(1)=\frac{o(1)}{\la^{\frac{\ell+\alpha+1}{2}}}.
\end{equation}
Next, from \eqref{eq-4.32} and \eqref{eq-4.33}, we get
\begin{equation}\label{eq-4.34}
\cos\left(\lambda\right)y_x(1)=\frac{o(1)}{\la^{\frac{\ell+\alpha-1}{2}}}\ \ \ \text{and}\ \ \
\sin\left(\lambda\right)y_x(1)=\frac{o(1)}{\la^{\frac{\ell+\alpha-1}{2}}}.
\end{equation}
It follows that 
$$
|\sin\left(\lambda\right)y_x(1)|^2+|\cos\left(\lambda\right)y_x(1)|^2=|y_x(1)|^2=\frac{o(1)}{\la^{\ell+\alpha-1}}.
$$
\textbf{Case 2}. If $b=k\pi$, $k\in \Z^\star$. Assume that $b=(2s+1)\pi$  the same argument for $b=2s\pi$. Then $A_1=A_2=0$. Using the expression of $e^{-B}$ and equation \eqref{eq-4.31}, we get 
\begin{eqnarray}
-\frac{b}{2\la^2}A_4y_x(1)&=&\frac{o(1)}{\la^{\frac{\ell+\alpha+1}{2}}},\label{eq-4.35}\\
\frac{1}{\la}A_4y_x(1)+i\frac{b^2}{8\la^2}A_3y_x(1)&=&\frac{o(1)}{\la^{\frac{\ell+\alpha+1}{2}}}.\label{eq-4.36}
\end{eqnarray}
Multiplying equation \eqref{eq-4.35} by $\la^2$, we get 
\begin{equation}\label{eq-4.37}
A_4y_x(1)=\frac{o(1)}{\la^{\frac{\ell+\alpha-3}{2}}}.
\end{equation}
Combining equations \eqref{eq-4.37} and \eqref{eq-4.36}, we get 
\begin{equation}\label{eq-4.38}
A_3y_x(1)=\frac{o(1)}{\la^{\frac{\ell+\alpha-5}{2}}}.
\end{equation}
Adding the squaring of equations \eqref{eq-4.37} and \eqref{eq-4.38}, then  using the fact $A_3^2+A_4^2=\sin^2\left(\frac{b}{2}\right)=1$, we get 
$$
|y_x(1)|^2=\frac{o(1)}{\la^{\ell+\alpha-5}}.
$$
The proof is thus complete.
\end{proof}
\noindent \textbf{Proof of Theorem \ref{Theorem-4.1}.} We divide the proof in two steps. \\

\noindent {\bf Step1. The energy decay estimation.} First, taking 
$$
\ell=\left\{\begin{array}{lll}
1-\alpha&\text{if}&b\notin \pi\Z,\\
5-\alpha&\text{if}&b\in \pi\Z
\end{array}
\right.
$$
in Lemma \ref{Theorem-4.4}, we get
\begin{equation*}
\left|y_x(1)\right|^2=o(1).
\end{equation*}
\noindent Then using equation \eqref{eq-4.27}, we deduce that $y_x(0)=o(1)$ and $u_x(0)=o(1)$. Then, by taking $h=x$ in equation \eqref{eq-4.22}, we get 
\begin{equation*}
\int_0^1|\la u|^2dx+\int_0^1|u_x|^2dx+\int_0^1|\la y|^2dx+\int_0^1|y_x|^2dx=o(1).
\end{equation*}
Hence $\|U\|_{\HH}=o(1)$, which contradicts \eqref{eq-4.2}, consequently
condition \eqref{(i)} holds. This implies, from Lemma \ref{bt}, the energy decay estimation \eqref{eq-4.1}. \\

\noindent {\bf Step 2. The optimality.}  Let $\varepsilon>0$ and set 
\begin{equation*}
S=\left\{\begin{array}{ll}
\displaystyle{1-\alpha-\epsilon,\quad \text{if }b\not\in  \pi\mathbb{Z}^\star},
\\ \\
\displaystyle{5-\alpha-\epsilon,\quad \text{if }b\in  2\pi\mathbb{Z}^\star},
\\ \\
\displaystyle{5-\alpha-\epsilon,\quad \text{if }b\in  \pi(2\mathbb{Z}+1)}.
\end{array}\right.
\end{equation*}
 For $|n|\geq n_0$, let 
\begin{equation*}
\lambda_n=\left\{\begin{array}{ll}
\displaystyle{\lambda_{1,n},\quad \text{if }b\not\in  \pi\mathbb{Z}^\star},
\\ \\
\displaystyle{\lambda_{1,n},\quad \text{if }b\in  2\pi\mathbb{Z}^\star},
\\ \\
\displaystyle{\lambda_{2,n},\quad \text{if }b\in  \pi(2\mathbb{Z}+1)},
\end{array}\right.
\end{equation*}
where $\left(\lambda_{1,n}\right)_{ |n|\geq n_0} $ and $\left(\lambda_{2,n}\right)_{ |n|\geq n_0} $ are the  simple eigenvalues of $\mathcal{A}$ defined in Proposition \ref{Theorem-3.2}. Moreover, let $U_{n}\in D(\mathcal{A})$ be the    normalized eigenfunction corresponding to $\lambda_n$. So, set the real sequence $(\beta_n)_{n\geq n_0}$ by 
\begin{equation*}
\beta_n=\left\{\begin{array}{ll}
n\pi+{\dfrac {\gamma \left( 1-\cos \left( b \right)  \right)\cos
 \left( \dfrac{\pi\alpha}{2} \right)}{ 2\left( n\pi \right) ^{1-\alpha}}}, \qquad \text{if } b\not\in  \pi\mathbb{Z}^\star, 
\\ \\
n\pi+{\dfrac {{b}^{2}}{8n\pi}}+{\frac {7b^4}{128{\pi}^{
3}{n}^{3}}}+{\dfrac {\gamma\,{b}^{6} \cos \left(\dfrac{\pi\alpha}{2}  \right)}{128\,{
\pi}^{5-\alpha}{n}^{5-\alpha}}},\quad \text{if }b\in  2\pi\mathbb{Z}^\star,
\\ \\
n\pi+\frac{\pi}{2}+\frac {{b}^{2}}{8n\pi}-\frac {{b}^{2}}{16\pi {n}^{2}}+\frac{{b}^{2}\left( 4{\pi}^{2}+7{b}^{2} \right) }{{128\pi}^{3}{n}^{3}}-\frac {{b}^{2} \left( 4{\pi}^{2}+21
{b}^{2} \right) }{256{\pi}^{3}{n}^{4}} + \frac {\gamma {b}^{6} 
\cos \left(\frac{\pi\alpha}{2} \right)  }{256 {\pi}^{5-\alpha}{n}^{5-\alpha}}
,\quad \text{if }b\in  \pi(2\mathbb{Z}+1),
\end{array}\right.
\end{equation*}
Therefore, we have
$$ (i\beta_{n}+\mathcal{A})U_{n}=(i\beta_{n}+\lambda_{n})U_{n},\quad\forall |n|\geq n_0. $$
It follows, from Proposition \ref{Theorem-3.2}, that
$$ 
\beta_{n}^{S}\| (i\beta_{n}I+\mathcal{A}U_{n} \|_{\mathcal{H}}\sim  \dfrac{C}{n^{\epsilon}},\quad \forall |n|\geq n_{0},
$$
where $C>0$. Thus, we deduce 
$$ 
\lim_{|n|\to +\infty}\beta_{n}^{S}\| (i\beta_{k}I+\mathcal{A})U_{n} \|_{\mathcal{H}}=0. 
$$
Finally, thanks to Lemma \ref{bt} (Theorem 2.4 in \cite{borichev:10}), we deduce that condition \eqref{(i)} false and consequently we cannot expect the energy decay rate $t^{-\frac{2}{S}}$. Therefore estimation \eqref{eq-4.1} is optimal.  The proof is thus complete.
\subsection{Polynomial Stability in the case $\mathbf{\eta>0}$ and  $\mathbf{a\neq 1}$}\label{Section-4.2}
\noindent In this part, we study the asymptotic behavior of solution of system \eqref{eq-2.1}-\eqref{eq-2.8} in the general case when $a\neq 1$. Our main result is the following Theorem.
\begin{theoreme}\label{Theorem-4.5}
Assume that $\eta>0$, $a\neq 1$ and that condition ${\rm (SC)}$ holds.Then, for any rational number $\sqrt{a}>0$ and almost all irrational number 
$\sqrt{a}>0$, there exists a constant $c>0$ independent of $U_0$, such that the following energy estimation
\begin{equation}\label{EnergyGeneral}
E(t)\leq \frac{C}{t^{\frac{2}{5-\alpha}}}\|U_0\|^2_{D(\AA)}
\end{equation}
holds for all initial data $U_0=(u_0,u_1,y_0,y_1,\omega_0)\in D(\AA)$.
In addition, estimation \eqref{EnergyGeneral} still be holds if  $a\in \Q$, $\sqrt{a}\not\in \Q$  and $b$ small enough.
\end{theoreme}
\noindent	Similar to Theorem \ref{Theorem-4.1}, we have to check condition \eqref{(i)} in Lemma \ref{Theorem-4.1} with $\ell=5-\alpha$.  By argument of contradiction, suppose that \eqref{(i)} false, then there exist a real sequence $(\la_n)$ and a sequence $U^n=(u^n,v^n,y^n,z^n,\omega^n)\in D(\mathcal{A})$, verifying the following conditions
	\begin{eqnarray}
	|\la_n|\longrightarrow+\infty,\quad \|U^n\|=\|(u^n,v^n,y^n,z^n,\omega^n)\|=1,\label{eq-4.39}\\
	\la_n^\ell(i\la_nI-\mathcal{A})U^n=(f_1^n,g_1^n,f_2^n,g_2^n,f_3^n)\longrightarrow0\quad\text{in}\quad \mathcal{H}\label{eq-4.40}.
	\end{eqnarray}
	Detailing equation \eqref{eq-4.40}, we get 
	\begin{eqnarray}
	i\la_nu^n-v^n&=&\frac{f_1^n}{\la_n^\ell}\longrightarrow 0\quad\text{in}\quad H_L^1(0,1),\label{eq-4.41}\\
	i\la_nv^n-u_{xx}^n+bz_n&=&\frac{g_1^n}{\la_n^\ell}\longrightarrow 0\quad\text{in}\quad L^2(0,1),\label{eq-4.42}\\
	i\la_ny^n-z^n&=&\frac{f_2^n}{\la_n^\ell}\longrightarrow0\quad\text{in}\quad H_0^1(0,1),\label{eq-4.43}\\
	i\la_nz^n-ay_{xx}^n-bv^n&=&\frac{g_2^n}{\la_n^\ell}\longrightarrow0\quad\text{in}\quad L^2(0,1),\label{eq-4.44}\\
(i\la_n+\xi^2+\eta)\omega^n(\xi)-v^n(1)\mu(\xi)&=&\frac{f^n_3(\xi)}{\la_n^\ell}\longrightarrow0\quad\text{in}\quad L^2(-\infty,+\infty).\label{eq-4.45}
	\end{eqnarray}
	For the simplicity, we dropped the index $n$.  Since the sequence $U_n$ is uniformly bounded in $\HH$, then using equation \eqref{eq-4.3}, we get 
\begin{equation}\label{eq-4.46}
-\gamma\kappa(\alpha)\int_{\R}(\xi^2+\eta)|\omega^n(\xi)|^2d\xi=\Re\left(\left<\left(i\la_nI-\AA\right)U^n,U^n\right>_{\HH}\right)=\frac{o(1)}{\la_n^{\ell}}.
\end{equation}
Similar to Lemma \ref{Theorem-4.2}, we can prove the following Lemma.
\begin{lemma}\label{Theorem-4.6} Assume that $\eta>0$, $a\neq 1$ and that condition ${\rm (SC)}$ holds. Then, the solution $(u,v,y,z,\omega)\in D(\AA)$ of \eqref{eq-4.4}-\eqref{eq-4.8} satisfies the following asymptotic behavior estimation
\begin{equation*}
\|u\|=\frac{O(1)}{\la},\ \|y\|=\frac{O(1)}{\la},\ |u_x(1)|=\frac{o(1)}{\la^{\frac{\ell}{2}}},\ |u(1)|=\frac{o(1)}{\la^{\frac{\ell+\alpha+1}{2}}}.
\end{equation*}
\end{lemma}
\noindent Similar to  Lemma \ref{Theorem-4.3}, we can prove the following Lemma.
\begin{lemma}\label{Theorem-4.6p} Let $\in W^{1,\infty}(0,1)$. Assume that $\eta>0$, $a\neq 1$ and that condition ${\rm (SC)}$ holds. Then, the solution $(u,v,y,z,\omega)\in D(\AA)$ of \eqref{eq-4.4}-\eqref{eq-4.8} satisfies the following asymptotic behavior estimation
\begin{equation}\label{eq-4.47}
-\int_0^1h'\left(|\la u|^2+|\la y|^2+|u_x|^2+a|y_x|^2\right)dx+h(1)|y_x(1)|^2-h(0)|y_x(0)|^2-h(0)|u_x(0)|^2=\frac{o(1)}{\la^{\ell-1+\alpha}}+\frac{O(1)}{\la}.
\end{equation}
\end{lemma}
\begin{lemma}\label{Theorem-4.7}
Let $\ell=5-\alpha$. Assume that $\eta>0$, $a\neq 1$ and that condition ${\rm (SC)}$ holds. Then, for any rational number $\sqrt{a}>0$ and almost all irrational number $\sqrt{a}>0$, the solution $U=(u,v,y,z,\omega)\in D(\AA)$ of system \eqref{eq-4.23}-\eqref{eq-4.24} satisfies the following estimation 
 \begin{equation}\label{eq-4.49}
   \left|y_x(1)\right|^2=o(1).
   \end{equation}
  In addition, estimation \eqref{eq-4.49} still be holds if $a\in \Q$, 
   $\sqrt{a} \not \in \Q$ and $b$ small enough.
\end{lemma}
\begin{proof}
By taking $h=x$ in \eqref{eq-4.47} we deduce that $y_x(1)$ is uniformly bounded. We will show that $\left|y_x(1)\right|^2=o(1)$ by contradiction argument. So, assume that $\left|y_x(1)\right|^2=1$. Eliminate $v$ and $z$ in equations \eqref{eq-4.41}-\eqref{eq-4.43} by \eqref{eq-4.42} and \eqref{eq-4.44}, we  obtain the reduced system
	\begin{eqnarray}
	\la^2u+u_{xx}-i\la by&=&-\dfrac{g_1+i\la f_1+bf_2}{\la^{\ell}},\label{eq-4.50}\\
	\la^2y+ay_{xx}+i\la bu&=&-\dfrac{g_2+i\la f_2-bf_1}{\la^{\ell}}\label{eq-4.51}.
	\end{eqnarray}
	Let $Y=(u,u_x,y,y_x)$, then system \eqref{eq-4.50} and \eqref{eq-4.51}, could be written as 
	\begin{equation}\label{eq-4.52}
	Y_x=BY+F,
	\end{equation}
	where 
	$$
	B=\begin{pmatrix}
	0&1&0&0\\
	-\la^2&0&i\la b&0\\
	0&0&0&1\\
	\frac{-i\la b}{a}&0&\frac{-\la^2}{a}&0
	\end{pmatrix},\quad F=\begin{pmatrix}
	0\\-\frac{g_1+i\la f_1+f_2}{\la^{\ell}}\\0 \\ -\frac{g_2+i\la g_1-bf_1}{\la^{\ell}}
	\end{pmatrix}\quad\text{and}\quad Y_0=\begin{pmatrix}
	0\\ u_x(0)\\ 0\\ y_x(0)
	\end{pmatrix}.
	$$
	The solution of equation \eqref{eq-4.52} at 1 is given by 
	\begin{equation}\label{eq-4.53}
	Y(1)=e^{B}Y_0+\int_0^1e^{B(x-z)}F(z)dz.
	\end{equation}
	Using Lemma \ref{Theorem-4.6}, we have 
\begin{equation}\label{eq-4.54}
Y(1)=\left(\frac{o(1)}{\la^{\frac{\ell+\alpha+1}{2}}},\frac{o(1)}{\la^{\frac{\ell}{2}}},0,1\right).
\end{equation}
	From equation \eqref{eq-4.53}, we obtain
	\begin{equation}\label{eq-4.55}
	e^{-B}Y(1)=Y_0+e^{-B}\int_0^1e^{B(1-z)}F(z)dz,
	\end{equation}
	where
	$$
	e^B=\begin{pmatrix}
	b_{11}&b_{12}&b_{13}&b_{14}\\
	b_{21}&b_{11}&b_{23}&b_{13}\\
	-\frac{b_{13}}{a}&-\frac{b_{14}}{a}&b_{33}&b_{34}\\
	-\frac{b_{23}}{a}&-\frac{b_{13}}{a}&b_{43}&b_{33}
	\end{pmatrix}\ \ \text{and}\ \ 
	e^{-B}=\begin{pmatrix}
	b_{11}&-b_{12}&b_{13}&-b_{14}\\
	-b_{21}&b_{11}&-b_{23}&b_{13}\\
	-\frac{b_{13}}{a}&\frac{b_{14}}{a}&b_{33}&-b_{34}\\
	\frac{b_{23}}{a}&-\frac{b_{13}}{a}&-b_{43}&b_{33},
	\end{pmatrix}
	$$
	and
	$$
	\left\{\begin{array}{lll}
	\vspace{0.15cm}b_{11}&=&\displaystyle
	\frac{((a-1)\la+\Delta)(e^{t_1}+e^{-t_1})+((1-a)\la+\Delta)(e^{t_2}+e^{-t_2})}{4\Delta},
\\ \noalign{\medskip}
	\vspace{0.15cm}b_{12}&=&\displaystyle
	\frac{a\sqrt{2}\left(((a-1)\la-\Delta)t_1\left(e^{-t_2}-e^{t_2}\right)+((a-1)\la+\Delta)t_2\left(e^{t_1}-e^{-t_1}\right)\right)}{4t_1t_2\Delta},
\\ \noalign{\medskip}
	\vspace{0.15cm}b_{13}&=&\displaystyle
	\frac{iab}{2\Delta}\left(e^{t_2}+e^{-t_2}-e^{t1}-e^{-t_1}\right),
\\ \noalign{\medskip}
	\vspace{0.15cm}b_{14}&=&\displaystyle
	\frac{-ia^2b\sqrt{2}}{2t_1t_2\Delta}\left(t_2\left(e^{t_1}-e^{-t_1}\right)-t_1\left(e^{t_2}-e^{-t_2}\right)\right),
\\ \noalign{\medskip}
	\vspace{0.15cm}b_{21}&=&\displaystyle
	\frac{-a\la\sqrt{2}\left(\left((a-1)\la^2+\la\Delta+2b^2\right)t_2\left(e^{t_1}-e^{-t_1}\right)+\left((a-1)\la^2+\la\Delta-2b^2\right)t_1\left(e^{t_2}-e^{-t_2}\right)\right)}{4t_1t_2\Delta},
\\ \noalign{\medskip}
	\vspace{0.15cm}b_{23}&=&\displaystyle
	\frac{iab\la\sqrt{2}\left(\left((a+1)\la+\Delta\right)t_2\left(e^{t_1}-e^{-t_1}\right)\right)+\left(\left(\Delta-(a+1)\la\right)t_1\left(e^{t_2}-e^{-t_2}\right)\right)}{4t_1t_2\Delta},
\\ \noalign{\medskip}
	\vspace{0.15cm}b_{33}&=&\displaystyle
	\frac{\left((1-a)\la+\Delta\right)\left(e^{t_1}+e^{-t_1}\right)+\left((a-1)\la+\Delta\right)\left(e^{t_2}+e^{-t_2}\right)}{4\Delta},
\\ \noalign{\medskip}
	\vspace{0.15cm}b_{34}&=&\displaystyle
	\frac{a\sqrt{2}\left(\left((a-1)\la+\Delta\right)t_1\left(e^{t_2}-e^{-t_2}\right)+\left((1-a)\la+\Delta\right)t_2\left(e^{t_1}-e^{-t_1}\right)\right)}{4t_1t_2\Delta},
\\ \noalign{\medskip}
	\vspace{0.15cm}b_{43}&=&\displaystyle
	\frac{\la\sqrt{2}\left(\left((1-a)\la^2-\la\Delta+2ab^2\right)t_1\left(e^{t_2}-e^{-t_2}\right)+\left((a-1)\la^2-\la\Delta-2ab^2\right)t_2\left(e^{t_1}-e^{-t_1}\right)\right)}{4t_1t_2\Delta},
	\end{array}
	\right.
	$$
	such that
	$$
	t_1=\frac{\sqrt{-2a\la\left((a+1)\la+\Delta\right)}}{2},\ \ t_2=\frac{\sqrt{2a\la\left(\Delta-(a+1)\la\right)}}{2}\ \ \text{and}\ \ \Delta=\sqrt{(a-1)^2\la^2+4ab^2}.
	$$
	Performing advanced calculation for the exponential of matrix $B$ and $-B$, we obtain the following matrix
	$$
	e^{B}=\begin{pmatrix}
	\cos(\la)&0&0&0\\
	-\la\sin(\la)-\frac{b^2}{2(a-1)}\cos(\la)&\cos(\la)&\frac{ib}{(a-1)}\left(a\sin(\la)-\sqrt{a}\sin\left(\frac{\la}{\sqrt{a}}\right)\right)&0\\
	0&0&\cos\left(\frac{\la}{\sqrt{a}}\right)&0\\
	\frac{-ib}{a(a-1)}\left(a\sin(\la)-\sqrt{a}\sin\left(\frac{\la}{\sqrt{a}}\right)\right)&0&-\frac{\la}{\sqrt{a}}\sin\left(\frac{\la}{\sqrt{a}}\right)+\frac{b^2}{2(a-1)}\cos\left(\frac{\la}{\sqrt{a}}\right)&\cos\left(\frac{\la}{\sqrt{a}}\right)
	\end{pmatrix}
	+(o_{ij}),
	$$
	where $o_{ij}=\frac{O(1)}{\la}$. In particular, we have
	$$
	o_{14}=\frac{iab}{(a-1)\la^2}\left(\sin(\la)+\sqrt{a}\sin\left(\frac{\la}{\sqrt{a}}\right)\right)+\frac{iab^3}{2(a-1)^2\la^3}\left(\cos(\la)-a\cos\left(\frac{\la}{\sqrt{a}}\right)\right)+\frac{O(1)}{\la^4},
	$$
	$$
	o_{31}=-\frac{ib}{(a-1)\la}\left(\cos\left(\frac{\la}{\sqrt{a}}\right)-\cos(\la)\right)+\frac{O(1)}{\la^2},
	$$
	$$
	o_{32}=-\frac{ib}{(a-1)\la^2}\left(\sin(\la)+\sqrt{a}\sin\left(\frac{\la}{\sqrt{a}}\right)\right)+\frac{O(1)}{\la^3}
	$$
	and
	$$
	o_{34}=\frac{\sqrt{a}}{\la}\sin\left(\frac{\la}{\sqrt{a}}\right)-\frac{ab^2}{2(a-1)\la^2}\cos\left(\frac{\la}{\sqrt{a}}\right)+\frac{O(1)}{\la^3}.
	$$	
	  Using the expression of $e^{B}$, $e^{-B}$, $F$ and Lemma \eqref{eq-4.54},  we get
    \begin{equation}\label{eq-4.56}
    e^{-B}\int_0^1e^{B(1-z)}F(z)dz=\left(\dfrac{o(1)}{\la^{\ell}},\dfrac{o(1)}{\la^{\ell}},\dfrac{o(1)}{\la^{\ell}},\dfrac{o(1)}{\la^{\ell}}\right) 
    \end{equation}
    and 
    \begin{equation}\label{eq-4.57}
    e^{-B}Y(1)=\left(o_{14}y_x(1)+\dfrac{o(1)}{\la^{\frac{\ell+\alpha+1}{2}}},\frac{O(1)}{\la}y_x(1)+\dfrac{o(1)}{\la^{\frac{\ell+\alpha-1}{2}}},o_{34}y_x(1)+\dfrac{o(1)}{\la^{\frac{\ell+\alpha+3}{2}}},O(1)y_x(1)+\dfrac{o(1)}{\la^{\frac{\ell+\alpha+1}{2}}}\right) .
    \end{equation}
Our aim is to show that $| y_x(1)| = o(1),$ suppose that $| y_x(1)| =1.$   Inserting equations \eqref{eq-4.56} and \eqref{eq-4.57} in \eqref{eq-4.55} and using the expression of $o_{14}$ and $o_{34}$,  we get 
    \begin{equation}\label{eq-4.58}
    \begin{split}
    \frac{iab}{(a-1)\la^2}\left(\sin(\la)+\sqrt{a}\sin\left(\frac{\la}{\sqrt{a}}\right)\right)+\frac{iab^3}{2(a-1)^2\la}\left(\cos(\la)-a\cos\left(\frac{\la}{\sqrt{a}}\right)\right)\\
    +\frac{O(1)}{\la^4}+\frac{o(1)}{\la^{\frac{\ell+\alpha+1}{2}}}+\frac{o(1)}{\la^{\ell}}=0
    \end{split}
    \end{equation}
    and 
    \begin{equation}\label{eq-4.59}
    -\frac{\sqrt{a}}{\la}\sin\left(\frac{\la}{\sqrt{a}}\right)+\frac{ab^2}{2(a-1)\la^2}\cos\left(\frac{\la}{\sqrt{a}}\right)+\frac{O(1)}{\la^3}+\frac{o(1)}{\la^{\frac{\ell+\alpha+3}{2}}}+\frac{o(1)}{\la^{\ell}}=0.
    \end{equation}
    Multiplying equations \eqref{eq-4.58} and \eqref{eq-4.59} by $\la^2$ and $-\frac{\la}{\sqrt{a}}$ respectively, we get 
    	\begin{equation}
    	\sin\left(\la+\frac{b^2}{2(a-1)\la}\right)=\frac{O(1)}{\la^2}+\frac{o(1)}{\la^{\frac{\ell+\alpha-3}{2}}}+\frac{o(1)}{\la^{\ell-2}}\label{eq-4.60}
    	\end{equation}
    	and
    	\begin{equation} 	
    	\sin\left(\frac{\la}{\sqrt{a}}-\frac{b^2\sqrt{a}}{2(a-1)\la}\right)=\frac{O(1)}{\la^2}+\frac{o(1)}{\la^{\frac{\ell+\alpha+1}{2}}}+\frac{o(1)}{\la^{\ell-1}}.\label{eq-4.61}
   	\end{equation}
Since $\ell=5-\alpha$, it follows from equations \eqref{eq-4.60} and \eqref{eq-4.61}, there exists $n,m\in \mathbb{Z}$ such that 
    	\begin{eqnarray}
  \lambda=n\pi  -\frac{b^2}{2\left(a-1\right)\lambda}  +o\left(\frac{1}{\lambda }\right),\label{eq-4.62}
     \\ \noalign{\medskip}
 \lambda    =m\pi\sqrt{a}+\frac{a b^2}{2\left(a-1\right)\lambda}+o\left(\frac{1}{\lambda^3 }\right).\label{eq-4.63}
    	\end{eqnarray}
    	Using the fact that $\la$ is big enough; i.e., $\la \sim \pi n\sim \pi\sqrt{a}m$, then by tacking the squares of equations \eqref{eq-4.62} and \eqref{eq-4.63}, we get respectively
    	\begin{eqnarray}
    	\la^2&=&n^2\pi^2-\frac{b^2}{a-1}+o(1),\label{eq-4.64}\\
    	\la^2&=&am^2\pi^2+\frac{ab^2}{a-1}+o\left(\frac{1}{\la^2}\right).\label{eq-4.65}
    	\end{eqnarray}
    	Combining equations \eqref{eq-4.64}-\eqref{eq-4.65}, we get 
    	\begin{equation}\label{eq-4.66}
    	n^2\pi^2-am^2\pi^2=b^2\left(\frac{a+1}{a-1}\right)+o(1).
    	\end{equation}
We distinguish three cases: \\[0.1in]    	
\textbf{Case 1.} Assume that $\sqrt{a}\in\mathbb{Q}.$ We have 
\begin{enumerate}
 \item If  $a=\frac{p_0^2}{q_0^2}=\frac{n^2}{m^2}$ where $p_0,q_0\in \Z$, then from equation \eqref{eq-4.66}, we get the following contradiction
    	$$
    	0=b^2\left(\frac{a+1}{a-1}\right)+o(1).
    	$$ 
\item If $a=\frac{p_0^2}{q_0^2}\neq \frac{n^2}{m^2}$ where $p_0,q_0\in \Z^\star$, then from equation \eqref{eq-4.66}, we get 
    	$$
    	n^2-\frac{p_0^2}{q_0^2}m^2=\frac{b^2}{\pi^2}\left(\frac{a+1}{a-1}\right)+o(1).
    	$$
    	Equivalently, we have
    	\begin{equation*}
    	\frac{nq_0-p_0m}{q_0}=\frac{b^2}{\pi^2}\left(\frac{a+1}{a-1}\right)\frac{q_0}{nq_0+p_0m}+\frac{o(1)}{\la}.
    	\end{equation*}
    	Then, we get the following contradiction 
    	$$
    	\frac{1}{q_0}\leq \frac{O(1)}{\la}+o(1).$$
 \end{enumerate}
 Consequently,   we get $|y_x(1)|=o(1)$ in the case $\sqrt{a}\in \mathbb{Q}$.\\[0.1in]
{\textbf{Case 2}}. Assume that $b$ is small enough, there exists $p_0,q_0\in \Z^\star$ such that $a=\frac{p_0}{q_0}$ and $a\neq \frac{p^2}{q^2}$ for all $p,q\in \Z^\star$. Then from equation \eqref{eq-4.66}, we have 
    	\begin{equation}\label{eq-4.67}
    	\left|\frac{q_0n^2-p_0m^2}{q_0}\right|\leq \frac{b^2}{\pi^2}\left(\frac{a+1}{a-1}\right)+o(1).
    	\end{equation}
    	Since $b$ is small enough, we can assume that 
    	\begin{equation}\label{eq-4.68}
    	b^2\leq \frac{\pi^2(a-1)}{2q_0(a+1)}.
    	\end{equation}
    	Consequently, using equations  \eqref{eq-4.67} and \eqref{eq-4.68}, we get the following contradiction
    	\begin{equation*}
    	\frac{1}{2q_0}\leq \frac{1}{q_0}-\frac{b^2(1+a)}{\pi(a-1)}\leq o(1).
    	\end{equation*}
 Therefore, we get $|y_x(1)|=o(1)$ in the case $a\in \mathbb{Q}$, $\sqrt{a}\not\in \mathbb{Q}$ and $b$ small enough.\\[0.1in]
\textbf{Case 3.} Assume that Cases 1 and 2 are not true, for almost real number $\sqrt{a}$,  subtracting    \eqref{eq-4.62} from \eqref{eq-4.63}, we get  
    \begin{equation}\label{eq-4.69}
 \frac{n}{m} -\sqrt{a}= \frac{(a+1) b^2}{2\left(a-1\right)\pi m \lambda}   +o\left(\frac{1}{\lambda }\right) . 
    \end{equation}	
  From     \eqref{eq-4.63}, we get
 \begin{equation}\label{eq-4.70}
 \frac{1}{\lambda}=\frac{1}{m\pi\sqrt{a}}+\frac{o\left(1\right)}{m^2 }.
 \end{equation}   	
 Inserting    \eqref{eq-4.70} in \eqref{eq-4.69}, we get	    \begin{equation}\label{eq-4.71}
 \frac{n}{m} -\sqrt{a}= \frac{(a+1) b^2}{2\sqrt{a}\left(a-1\right)\pi^2 m^2 }   +\frac{o\left(1\right)}{m^2 } .
    \end{equation}
From Theorem 1.10 in  \cite{Bugeaud01}, we have for almost all real numbers $\xi$  there exists infinitely many integers $n,\ m$ such that
    \begin{equation}\label{eq-4.72}
\left| \xi-\frac{n}{m}    \right|<\frac{1}{m^2 \ln|m|}. 
    \end{equation}
Let 	$\xi=\sqrt{a}$, then from \eqref{eq-4.71} and \eqref{eq-4.72} there exist infinitely many integers $n,\ m$ such that
 \begin{equation}\label{eq-4.73}
\left|\frac{(a+1) b^2}{2\sqrt{a}\left(a-1\right)\pi^2 m^2 }   +\frac{o\left(1\right)}{m^2 }  \right|=\left|\frac{n}{m} -\sqrt{a}\right|<\frac{1}{m^2 \ln|m|}.
\end{equation}   	
 Since $m\sim \lambda$, then estimation   \eqref{eq-4.73} can be written as   	 \begin{equation*}
\left|\frac{(a+1) b^2}{2\sqrt{a}\left(a-1\right)\pi^2  }   +o(1) \right|=o(1).
\end{equation*}    	
  Equivalently, we have 
  \begin{equation*}
\left|\frac{(a+1) b^2}{2\sqrt{a}\left(a-1\right)\pi^2  } \right|=o(1)
\end{equation*} 
 and this a contradiction.  Therefore, we get $|y_x(1)|=o(1)$ for almost real number $\sqrt{a}$.\\[0.1in]
\noindent Hence, in the three cases, we get \eqref{eq-4.49}. The proof is thus complete.
 \end{proof}
 \noindent \textbf{Proof of Theorem \ref{Theorem-4.5}.} 
 \noindent Let $\ell=5-\alpha$. By tacking $h=1$ in equation \eqref{eq-4.47}, we get 
\begin{equation}\label{eq-4.48}
\left|y_x(1)\right|^2-\left|y_x(0)\right|^2-\left|u_x(0)\right|^2=\frac{o(1)}{\la^{\ell-1+\alpha}}.
\end{equation}
It follows from, Lemma \ref{Theorem-4.7}, that  
 \begin{equation}\label{eq-4.74}
\left|y_x(1)\right|^2=\left|y_x(0)\right|^2=\left|u_x(0)\right|^2=o(1).
\end{equation}
 Finally, by taking $h=x$ in equation \eqref{eq-4.47} and using \eqref{eq-4.74}, we get $\|U\|_{\mathcal{H}}=o(1)$ which contradicts\eqref{eq-4.39}, consequently condition \eqref{(i)} holds. This implies,from Lemma \ref{bt}, the energy decay estimation \eqref{EnergyGeneral}. The proof is thus complete.
\section*{Conclusion}
\noindent We have studied the influence of the coefficients on the indirect stabilization of a system of two wave  equations coupled by velocities, with only one  fractional derivative control. In this work, we consider the Caputo's fractional derivative of order $\alpha\in ]0,1[$ and $\eta\geq 0$. If the wave speeds are equal (i.e., $a=1$), $\eta>0$ and if the coupling parameter $b=k\pi $ (resp. $ b \neq k\pi$), $k\in \Z$ and it is outside a discrete set of exceptional values, a non-uniform stability is expected. Then, using a frequency domain approach combining with a multiplier method, we have proved an optimal polynomial energy decay rate of type ${t^{-\frac{2}{1-\alpha}}}$ (resp. ${t^{-\frac{2}{5-\alpha}}}$). In the general case, when $a\neq 1$ a non uniform stability is expected. Finally, if $\sqrt{a}$ is a rational number or ($a$ is a rational number and $b$ is small enough or for almost irrational number $\sqrt{a}$) and if $b$ is outside another discrete set of exceptional values, using a frequency domain approach, we proved a polynomial energy decay rate of type ${t^{-\frac{2}{5-\alpha}}}$. We conjecture that the remaining cases could be analyzed in the same way with a slower polynomial decay rate.

\end{document}